\documentclass[bibliography=toc, DIV=12]{scrartcl}

\usepackage[utf8]{inputenc}
\usepackage[english]{babel}

\usepackage{lmodern}
\usepackage{amssymb}
\usepackage{amsmath}
\usepackage{mathtools}
\usepackage{enumitem}
\usepackage{booktabs}
\usepackage{subcaption}

\usepackage[svgnames]{xcolor}
\usepackage{tikz}
\usetikzlibrary{cd, backgrounds, patterns}
\usepackage{forest}

\tikzcdset{scale cd/.style={every label/.append style={scale=#1},
      cells={nodes={scale=#1}}}}

\usepackage[autostyle]{csquotes}
\usepackage[style=alphabetic, sorting=nyt, maxnames=10, maxalphanames=5]{biblatex}
\let\oldmultinamedelim\multinamedelim
\let\oldfinalnamedelim\finalnamedelim
\RenewDocumentCommand{\multinamedelim}{}{--}
\RenewDocumentCommand{\finalnamedelim}{}{--}
\AtBeginBibliography{
  \RenewDocumentCommand{\multinamedelim}{}{\oldmultinamedelim}
  \RenewDocumentCommand{\finalnamedelim}{}{\oldfinalnamedelim}
}

\usepackage{hyperref}
\hypersetup{colorlinks, urlcolor=blue!65!black, citecolor=green!65!black, linkcolor=red!65!black}
\usepackage{amsthm, thmtools}

\numberwithin{equation}{section}
\declaretheorem[sibling = equation]{theorem}
\declaretheorem[sibling = equation]{proposition}
\declaretheorem[sibling = equation]{corollary}
\declaretheorem[sibling = equation]{lemma}

\declaretheorem[sibling = equation, style = definition]{definition}
\declaretheorem[sibling = equation, style = definition]{convention}
\declaretheorem[sibling = equation, style = definition]{notation}
\declaretheorem[sibling = equation, style = remark]{example}
\declaretheorem[sibling = equation, style = remark]{remark}

\declaretheorem[name = Theorem]{theoremIntro}

\declaretheorem[numbered = no, name = Remark, style = remark]{remarkIntro}

\DeclareMathOperator{\id}{id}
\DeclareMathOperator{\pr}{pr}
\DeclareMathOperator{\sgn}{sgn}
\DeclareMathOperator{\Hom}{Hom}

\NewDocumentCommand{\opP}{}{\mathcal{P}}
\NewDocumentCommand{\opQ}{}{\mathcal{Q}}

\NewDocumentCommand{\C}{}{\mathcal{C}}
\NewDocumentCommand{\SC}{}{\mathcal{SC}}
\NewDocumentCommand{\vor}{}{\mathrm{vor}}
\NewDocumentCommand{\SCv}{}{\SC^{\vor}}
\NewDocumentCommand{\Ass}{}{\mathcal{A}\mathrm{ss}}

\NewDocumentCommand{\Top}{}{\mathsf{Top}}
\NewDocumentCommand{\Ch}{}{\mathsf{Ch}}

\NewDocumentCommand{\colc}{}{\mathfrak{c}}
\NewDocumentCommand{\colo}{}{\mathfrak{o}}

\NewDocumentCommand{\Q}{}{\mathbb{Q}}
\NewDocumentCommand{\K}{}{\mathbb{K}}

\NewDocumentCommand{\cK}{}{\mathcal{K}}
\RenewDocumentCommand{\S}{}{\mathbb{S}}

\NewDocumentCommand{\OurOperators}{m s t- o m}{{
\IfBooleanTF{#2}{\hat{#1}}{
  \IfBooleanTF{#3}{\bar{#1}}{
    #1}}
^{\ifx#5Q \mathcal{Q} \else #5 \fi}
\IfValueT{#4}{_{#4}}
}}

\def\pro{\OurOperators{\mu}}
\def\inc{\OurOperators{\iota}}
\def\act{\OurOperators{\alpha}}
\def\push{\OurOperators{\beta}}
\def\loo{\OurOperators{\ell}}
\def\spec{\OurOperators{\eta}}
\def\gam{\OurOperators{\gamma}}

\title{Non-formality of Voronov's Swiss-Cheese operads}
\author{Najib Idrissi\thanks{Université Paris Cité and Sorbonne Université, CNRS, IMJ-PRG, F-75013 Paris, France.} \and Renato Vasconcellos Vieira\thanks{Universidade de São Paulo, ICMC, São Carlos, Brasil.}}
\date{September 2023}

\begin{document}

\maketitle

\begin{abstract}
  The Swiss-Cheese operads, which encode actions of algebras over the little $n$-cubes operad on algebras over the little $(n-1)$-cubes operad, comes in several variants.
  We prove that the variant in which open operations must have at least one open input is not formal in characteristic zero.
  This is slightly stronger than earlier results of Livernet and Willwacher.
  The obstruction to formality that we find lies in arity $(2, 2^n)$, rather than $(2, 0)$ (Livernet) or $(4, 0)$ (Willwacher).
\end{abstract}

\tableofcontents

\addsec{Introduction}

The little cubes operads $\C_n$ (for $n \geq 1$), introduced by \textcite{BoardmanVogt1968,May1972}, are topological operads that encode strongly homotopy commutative (up to degree $n$) algebras.
The space $\C_n(r)$ is composed of configurations of $r$ rectilinear $n$-cubes embedded in the unit $n$-cube, with pairwise disjoint interiors.

To prove his celebrated formality theorem, \textcite{Kontsevich1999,Kontsevich2003} used the Swiss-Cheese operads $\SC_{n+1}$ of \textcite{Voronov1999}, which encode central actions of $\C_{n+1}$-algebras on $\C_n$-algebras.
The space $\SC_{n+1}(r,s) \subseteq \C_{n+1}(r+s)$ consists of configurations of rectilinear cubes of two different kinds: $r$ ``open'' ones, whose bottom faces must be included in the bottom face of the ambient cube; and $s$ ``closed'' ones, which have no requirements.
This operad was e.g. used to define the generalized Hochschild complex of a $\C_n$-algebra for $n \geq 2$~\cite[Definition~9]{Kontsevich1999}.

A fundamental property of $\C_n$ is its formality, which is a notion originating from rational homotopy theory~\cite{Sullivan1977}.
Briefly, an operad $\opP$ is called formal over a field $\K$ if the dg-operad $C_*(\opP; \K)$ is quasi-isomorphic to its homology $H_*(\opP; \K)$.
It has been shown using several different methods~\cite{Kontsevich2003,Tamarkin2003,Petersen2014,FresseWillwacher2015,BritoHorel2019} that $\C_n$ is formal over $\Q$ for any $n \geq 2$.
In constrast, $\C_n$ is not formal as a (symmetric) operad over $\mathbb{F}_p$ for any prime $p$ and $n \geq 2$~\cite[Remark~6.9]{CiriciHorel2018}.
It is also not formal over $\mathbb{F}_2$ as a non-symmetric operad~\cite{Salvatore2018}.

The question of formality for the Swiss-Cheese operads is subtler.
\textcite{Voronov1999} originally defined an operad $\SCv_{n+1} \subseteq \SC_{n+1}$ such that $\SCv_{n+1}(0, s)$ is empty for all $s \geq 0$, i.e., open operations must have open inputs.
Today, the Swiss-Cheese operad $\SC_{n+1}$ commonly allows such operations, with $\SC_{n+1}(0, s) = \C_{n+1}(s)$.
Roughly speaking, if $\SCv_{n+1}$ encodes a central action $A \otimes B \to B$ of a $\C_{n+1}$-algebra $A$ on a $\C_{n-1}$-algebra $B$, the larger operad $\SC_{n+1}$ encodes a central morphism $f:A \to B$~\cite{HoefelLivernet2012}.
Given a central morphism, the action is given by $a \cdot b \coloneqq f(a) b$.
In the presence of units, $f$ can be recovered from the action by $f(a) \coloneqq a \cdot 1_B$.
The two notions are not equivalent in general.

\textcite{Livernet2015} proved, using the theory of operadic Massey products, that $\SC_{n+1}$ is not formal over any field of characteristic different from $2$.
\textcite{Willwacher2017a} proved that the non-formality of $\SC_{n+1}$ is equivalent to the relative non-formality of the standard inclusion of operads $\C_{n-1} \to \C_{n}$~\cite{TurchinWillwacher2018}.
However, both proofs use the elements of the space $\SC_{n+1}(0, s)$ in an essential way and they thus cannot be applied to $\SCv_{n+1}$.
Non-formality of $\SC_{n+1}$ does not imply non-formality of its suboperad $\SCv_{n+1}$.
We close this gap:

\begin{theoremIntro}
  \label{thmA}
  Voronov's Swiss-Cheese operad $\SCv_{n+1}$ is not formal over any field of characteristic different from $2$ for any $n \geq 1$.
\end{theoremIntro}

\begin{remarkIntro}
  The case $n = 1$ of Theorem~\ref{thmA} was proved in an appendix of the second-named author's PhD thesis~\cite{Vieira2018a}.
\end{remarkIntro}

\begin{remarkIntro}
  If $n = 0$, then $\SC_1$ is formal over any ring.
  We give a more detailed explanation in Section~\ref{sec:formality nonformality}.
\end{remarkIntro}

The homotopy type of the Swiss-Cheese operad is of interest in deformation quantization (see \textcite{Kontsevich1999}).
To carry out such applications, it is important to obtain small combinatorial models (in the sense of rational homotopy theory) for the operads involved.
Our non-formality result implies that, even in the case of Voronov's Swiss-Cheese operad, such models must have a nontrivial differential.
Some models are known for the full Swiss-Cheese operad $\SC_{n+1}$~\cite{Willwacher2015a,Idrissi2017} and they can be truncated to give a model of $\SCv_{n+1}$.
However, it is conceivable that models of $\SCv_{n+1}$ that are simpler than these truncated models exist.
Our result (and the explicit bound on the arity of the obstruction, see below) shows that such models cannot simpler starting at arity $(2, 2^n)$.

\paragraph{Proof strategy}

Our proof is similar in spirit to that of \textcite{Livernet2015}, which uses the theory of Massey products.
These encode the idea that, when a product of three cohomology classes $x y z$ is zero in two different ways because $xy = yz = 0$, one may define some new class $\langle x, y, z \rangle$, and it constitutes an obstruction to formality if nonzero.
This theory has been extended by Livernet to operads.

\begin{figure}[htbp]
  \centering
  \includegraphics[width=\linewidth]{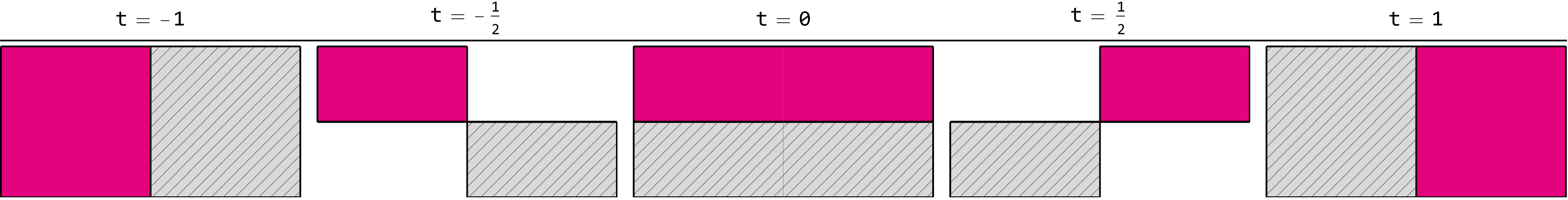}
  \caption{The path used in \cite{Livernet2015} for $n=2$. The hatched square is an open input.}
  \label{fig:livernet-proof}
\end{figure}

Livernet then applied this criterion to $H_*(\SC_{n+1})$.
Figure~\ref{fig:livernet-proof} illustrates that in an $H_*(\SC_2)$-algebra $(A, B, f)$ where $A$ is an $H_*(\C_2)$-algebra, $B$ is an $H_*(\C_1)$-algebra, and $f : A \to B$ is a morphism, the map $A \otimes B \to B, \, a \otimes b \mapsto f(a) \cdot b$ is equal to $a \otimes b \to b \cdot f(a)$.
It thus follows that the map $a_1 \otimes a_2 \mapsto f(a) f(b) - f(b) f(a)$ vanishes in two different ways.
The homotopies that witness this vanishing glue to give a nontrivial homology class in $H_1(\SC_2(2, 0))$, i.e., a nonzero Massey product.
The proof for higher $n$ is similar (see \cite[Section~4]{Livernet2015}).
Our goal, in this paper, is to construct something analogous in the chain complex of $\SCv_{n+1}$.

The main hurdle we clear, in this paper, is the combinatorial difficulty of the proof.
The path illustrated in Figure~\ref{fig:livernet-proof} uses four affine paths to construct a half-circle in $\SC_2(2,0)$.
More generally, in $\SC_{n+1}(2,0)$, the nontrivial $(n-1)$-sphere is constructed out of two hemispheres, which are in turn constructed by gluing just four $(n-1)$-cubes.
However, in $\SCv_{n+1}$, we cannot use such simple paths: all the operations that we use must have at least one open input, and as we compose them together, the number of inputs grows quickly.
In order to handle the induced complexity, we work by induction, starting with the proof of the non-formality of $\SCv_2$ (initially found in~\cite{Vieira2018a}) which involves constructing a half-circle out of $2 \times 4$ affine paths (see Figure~\ref{fig:eta}, in which each path of the type $\push[\pm]{2}$ is already a concatenation of two paths).
Then, in a given dimension $n$, we build an $(n-1)$-sphere $\S^{n-1} \subset \SCv_{n+1}(2, 2^n)$ out of the $(n-2)$-chains constructed in $\SCv_{n-1}(2, 2^{n-1})$ by composing them with affine paths.
This allows us to express the $(n-1)$-sphere $\S^{n-1}$ as the suspension of $\S^{n-2}$.

The proof of \textcite{Livernet2015} shows that there is an obstruction to formality in $\SC_{n+1}(2, 0)$ for all $n$, while that of \textcite{Willwacher2017a} finds one in $\SC_{n+1}(4, 0)$ for all $n$.
With our methods, the obstruction to formality lives in $\SCv_{n+1}(2, 2^n)$.
It would be an interesting question to determine if an obstruction can be found in a lower-arity component of Voronov's Swiss-Cheese operad, or alternatively to question if some arity-truncation of $\SCv_{n+1}$ is formal.

In addition to the above combinatorial difficulty of the proof, we found that using cubical chains (rather than the more commonly used simplicial chains) was beneficial.
This leads to a couple of interesting questions about the homotopy theory of operads in cubical $\omega$-groupoids, detailed in Section~\ref{sec background}.

\paragraph{Acknowledgments}

N.I.\ thanks Muriel Livernet and Thomas Willwacher for helpful discussions.
N.I.\ was supported by the project HighAGT (ANR-20-CE40-0016), the project SHoCoS (ANR-22-CE40-0008), and the IdEx Université Paris Cité (ANR-18-IDEX-0001).

R.V.V.\ was financed by the grant 2020/06159-5, São Paulo Research Foundation (FAPESP).

The authors thank the Réseau Franco-Brésilien de Mathématiques (GDRI-RFBM) for their mobility grant.

\section{Background, conventions, and notation}
\label{sec background}

Let us begin by recalling some background necessary for the statement and the proof of our theorem.
While we claim no originality in this section, we advise the reader that we introduce several notation that will be used throughout the text.

Given an integer $n \geq 0$, let $[n] \coloneqq \{0, \dots, n\}$ and $\underline{n} \coloneqq \{1, \dots, n\}$.
We work in $\Top$, the category of topological spaces and continuous maps, and $\Ch(\K)$ the category of nonnegatively graded chain complexes over some fixed field $\K$ of characteristic different from $2$.

\subsection{Cubical homotopy theory}
\label{sec cubical}

We will be working extensively with little cubes operads (see Section~\ref{sec cubes operads}).
The homology classes that we will construct will thus be easier to describe using \emph{cubical} chains, rather than using the more common simplicial chains.
We recall certain features (such as connections) that are lesser known.
See~\cite{BrownHigginsSivera2011} for a general reference.

\begin{remark}
  Instead of using the standard unit $k$-cube $[0,1]^k$, we will use the $k$-cube $I = [-1, 1]^k$, as this will make all of our formulas much simpler.
  This, of course, makes no material difference in what follows.
\end{remark}

\begin{definition}
  Given a topological space $X$, denote $K_k X = \operatorname{Map} \bigl( [-1, 1]^k, X \bigr)$ and let $C'_k X$ be the $\K$-space spanned by $K_k X$.
  Given a basis element $\sigma \in K_k X$ and integers $i \in \underline{k}$, $j \in \underline{k+1}$, the faces $d_i^+ \sigma, d_i^- \sigma \in C'_{k-1} X$ and degeneracies $s_j \sigma \in C'_{k+1} X$ are given by:
  \begin{align}
    (d_i^+ \sigma)(t_1, \dots, t_{k-1}) & = \sigma(t_1, \dots, t_{i-1}, 1, t_i, \dots, t_k); \\
    (d_i^-\sigma)(t_1, \dots, t_{k-1})  & = \sigma(t_1, \dots, t_{i-1}, -1, t_i, \dots, t_k); \\
    (s_j\sigma)(t_1, \dots, t_{k+1})    & = \sigma(t_1, \dots, t_{j-1}, t_{j+1}, \dots, t_{k+1}).
  \end{align}
\end{definition}

\begin{remark}
  The above operations satisfy obvious identities that are reminiscent of simplicial identities.
  There is a general notion of \emph{cubical object}, which specializes to cubical sets, cubical $\K$-spaces (the notion defined above), cubical topological spaces, and so on.
  The collection $\{K_\bullet X\}$ forms a cubical set, while $C'_\bullet X$ form a cubical $\K$-space.
\end{remark}

The faces of a cube are often depicted pictorially as follows, for $k = 1$ or $k=2$:
\begin{equation*}
  \begin{tikzcd}[arrows={dash}]
    d_1^-\sigma \ar[r, "\sigma \in C_1 X"] & d_1^+\sigma
  \end{tikzcd}
  \quad
  \begin{tikzcd}[arrows={dash}]
    d_1^+ d_1^- \sigma \ar[r, "d_2^+ \sigma"]
    & d_1^+ d_1^+ \sigma
    \\
    d_1^- d_1^- \sigma \ar[u, "d_1^-\sigma"] \ar[r, "d_2^-\sigma" swap]
    \ar[ur, phantom, "\sigma \in C_2 X"]
    & d_1^- d_1^+ \sigma \ar[u, "d_1^+ \sigma" swap]
  \end{tikzcd}
\end{equation*}

\begin{definition}
  Let $C_k X$ be the quotient of $C'_k X$ by degenerate cubes.
  The differential $d : C_k X \to C_{k-1} X$ is given by the signed sum of faces, $d = \sum_{i=1}^k (-1)^i (d_i^- - d_i^+)$.
  We thus obtain a complex, $(C_* X, d)$, called the \emph{cubical singular chain complex} of $X$.
  The homology of this complex is denoted by $H_* X = \{ H_k X \}_{k \geq 0}$, which we can also view as a chain complex with vanishing differential.
\end{definition}

Cubical singular homology coindices with simplicial singular homology~\cite{EilenbergMacLane1953}.
Cubical homology has a great advantage over simplicial homology: the product of two cubes is a cube.
This implies, for example, that the cubical Eilenberg--Zilber map is an isomorphism, rather than merely a homotopy equivalence.
However, the Dold--Kan equivalence between chain complexes and cubical abelian groups is missing~\cite[Remark~14.8.3]{BrownHigginsSivera2011}.

To recover the equivalence, one has to consider instead cubical abelian groups equipped with \emph{connections}, which generalize degeneracies (see \cite[Section 13.1]{BrownHigginsSivera2011}).
A degenerate $k$-cube $s_j \sigma$ can be thought of as a ``thin cube'' which is constant in the direction of the $j$th coordinate.
Connections provide other kinds of thin cubes.
See Figure~\ref{fig thin squares} for examples of thin squares obtained from a segment.

\begin{figure}[htbp]
  \centering
  \subcaptionbox{Vertical thin square.}[.32\linewidth]{
    \begin{tikzcd}[arrows={dash}, ampersand replacement=\&]
      \sigma(1) \ar[r, dashed, "\sigma(1)"]
      \& \sigma(1)
      \\
      \sigma(-1) \ar[u, "\sigma"] \ar[r, dashed, "\sigma(-1)" swap]
      \ar[ur, phantom, "s_1 \sigma"]
      \& \sigma(-1) \ar[u, "\sigma" swap]
    \end{tikzcd}
  }
  \subcaptionbox{Horizontal thin square.}[.32\linewidth]{
    \begin{tikzcd}[arrows={dash}, ampersand replacement=\&]
      \sigma(-1) \ar[r, "\sigma"]
      \& \sigma(1)
      \\
      \sigma(-1) \ar[u, dashed, "\sigma(-1)"] \ar[r, "\sigma" swap]
      \ar[ur, phantom, "s_2 \sigma"]
      \& \sigma(1) \ar[u, dashed, "\sigma(1)" swap]
    \end{tikzcd}
  }
  \subcaptionbox{L-shaped thin square.}[.32\linewidth]{
    \begin{tikzcd}[arrows={dash}, ampersand replacement=\&]
      \sigma(1) \ar[r, dashed, "\sigma(1)"]
      \& \sigma(1)
      \\
      \sigma(-1) \ar[u, "\sigma"] \ar[r, "\sigma" swap]
      \ar[ur, phantom, "\Gamma_1 \sigma"]
      \& \sigma(1) \ar[u, dashed, "\sigma(1)" swap]
    \end{tikzcd}
  }
  \caption{Examples of thin squares obtained from a one-dimensional segment.}
  \label{fig thin squares}
\end{figure}


\begin{definition}[{\cite[Definition 13.1.3]{BrownHigginsSivera2011}}]
  A cubical object \emph{with connections} is a cubical object $K_\bullet$ equipped with morphisms $\Gamma_j : K_{k-1} \to K_k$ satisfying natural identities.
\end{definition}

\begin{example}
  Given a space $X$, the cubical set $K_\bullet X$ has connections given by (for $\sigma : [-1, 1]^k \to X$ and $j \in \underline{k}$):
  \begin{equation}
    (\Gamma_j \sigma)(t_1, \dots, t_{k+1}) \coloneqq \sigma(t_1, \dots, t_{j-1}, \max(t_j, t_{j+1}), t_{j+2}, \dots, t_{k+1}).
  \end{equation}
\end{example}

We will not list every identity here, as they all come from the dual ones satisfied by the prototypical example $[-1, 1]^\bullet$.
The most important ones for us are:
\begin{equation}
  d_j^- \Gamma_j \sigma = d_j^- \Gamma_{j+1} \sigma = \sigma;
  \;
  d_j^+ \Gamma_j \sigma = d_j^+ \Gamma_{j+1} \sigma = s_j d_j^+ \sigma;
  \;
  \Gamma_j s_j \sigma = s_j^2 \sigma = s_{j+1} s_j \sigma.
\end{equation}

The last piece of data we will need is that of compositions.

\begin{definition}[{\cite[Definitions 13.1.7, 13.2.1]{BrownHigginsSivera2011}}]
  A cubical $\omega$-groupoid is a cubical object with connections $K_\bullet$ equipped with partial compositions ${+_i} : K_n \times K_n \to K_n$ and inversions ${-_i} : K_n \to K_n$ such that $x +_i y$ is defined if $d_i^+ x = d_i^- y$, which satisfy several compatibility axioms, and such that each $(K_n, {+_i}, {-_i})$ defines a groupoid (with source and target maps $s_i d_i^+$ and $s_i d_i^-$).
\end{definition}

\begin{example}[Prototypical example]
  The cubical set $K_\bullet X$ of a space $X$ has compositions and inversions given by (for $d_i^+ a = d_i^- b$):
  \begin{align}
    (a +_i b)(t_1, \dots, t_n) & \coloneqq
    \begin{cases}
      a(t_1, \dots, 2t_i+1, \dots, t_n), & t_i \leq 0; \\
      b(t_1, \dots, 2t_i-1, \dots, t_n), & t_i \geq 0; \\
    \end{cases}
    \\
    ({-_i} a)(t_1, \dots, t_n) & \coloneqq a(t_1, \dots, -t_i, \dots, t_n).
  \end{align}
  However, it does not define a strict groupoid, as $a +_i ({-_i} a)$ is not equal to the unit.
\end{example}

\begin{remark}[Consequence of {\cite*[Lemma 14.8.2]{BrownHigginsSivera2011}}]
  In an abelian category (e.g. $\K$-spaces), a cubical object $K_\bullet$ with connections is canonically an $\omega$-groupoid, where (for $a,b \in K_n$ such that $d_i^+ a = d_i^- b \eqqcolon x$):
  \begin{align}
    a +_i b & \coloneqq a - s_i x + b, &
    {-_i} a & \coloneqq s_i x - a.
  \end{align}
\end{remark}

\begin{theorem}[Immediate consequence of {\cite[Theorem~14.8.1]{BrownHigginsSivera2011}}]\label{thm dold-kan}
  There is an equivalence of categories between chain complexes and cubical $\omega$-groupoids in $\K$-spaces.
  The equivalence $\cK : \Ch(\K) \leftrightarrows \omega \mathsf{Gpd}_{\K} : N$ acts on objects by:
  \begin{align}
    (\cK C)_\bullet & \coloneqq \Hom_{\Ch(\K)} \bigl( C_*([-1, 1]^\bullet), C \bigr), \\
    (N K)_*         & \coloneqq \biggl( (NK)_n =  K_n / \sum_{i=1}^n s_i(K_{n-1}), \; d = \sum_{i=1}^n (-1)^i (d_i^- - d_i^+) \biggr).
  \end{align}
\end{theorem}

\subsubsection{Some useful maps}
\label{sec useful maps}

Let us now define some maps between cubes that will prove useful in Section~\ref{sec non-formal all}.

\begin{definition}\label{def Phi}
  Let $n \geq 1$ be an integer and $S,T\subseteq\underline{n}$ be subsets with $S\cap T=\emptyset$ and $T \neq \emptyset$.
  Let $\Phi_{S,T}:[-1,1]^n\rightarrow [-1,1]^{\underline{n}\setminus T}$ such that for $t\in [-1,1]^n$ and $j \in \underline{n} \setminus T$, we have:
  \[
    \Phi_{S,T} (t)_j
    \coloneqq
    \begin{cases}
      \min(t_j,-\max\{t_i\mid i\in T\}),
       & \text{if } j\in S;
      \\
      \max(t_j,-\max\{t_i\mid i\in T\}),
       & \text{if } j\not\in S.
    \end{cases}
  \]
\end{definition}

\begin{example}
  Let $n = 4$, $S = \{1\}$ and $T = \{2,4\}$. Then we have:
  \[
    \Phi_{\{1\}, \{2,4\}}(t) =
    \begin{pmatrix}
      \min (t_1, -\max (t_2,t_4)) &
      \cdot                       &
      \max (t_3, -\max (t_2,t_4)) &
      \cdot
    \end{pmatrix}.
  \]
\end{example}

The maps $\Phi_{S,T}$ defined above are composites of cofaces, coconnections, and groupoids cooperations, and the diagonal map $\Delta : [-1,1] \to [-1,1]^2, \; t \mapsto (t,t)$.
The cubical analogue of the Alexander--Whitney map $\triangle : C_*(A_\bullet \times B_\bullet) \to C_*(A_\bullet) \otimes C_*(B_\bullet)$ provides a cubical approximation of the diagonal.
We will not need its explicit formula, which  can be found using the technique of acyclic models (see~\cite{EilenbergMacLane1953,EilenbergMacLane1953a,EilenbergZilber1953} where this is worked out for simplicial sets).
We can thus define operations induced (contravariantly) by $\Phi_{S, T}$ on any cubical $\omega$-groupoid $K_\bullet$, replacing occurrences of the diagonal by $\triangle$.

\begin{convention}\label{conv Phi}
  To ease notation, we will simply denote the above operations on a cubical $\omega$-groupoid $K_\bullet$ by $K_{n - |T|} \to K_n, \; x \mapsto x \Phi_{S, T}$, even though they are not literally obtained by precomposition with $\Phi_{S,T}$ in general.
\end{convention}

The following maps will also be useful.

\begin{definition}
  Let $n\geq 1$ be an integer. Let $\Psi_n:[-1,1]^{n+1}\rightarrow [-1,1]^n$ such that for $t\in [-1,1]^{n+1}$ and $j\in\underline n$ we have:
  $$
    \Psi_n(t)_j\coloneqq\begin{cases}
      t_j,               & \text{if } j<n;   \\
      \min (t_n,t_{n+1}) & \text{if } j = n.
    \end{cases}
  $$
  Note that if $X$ is a topological space and $\sigma \in K_n X$, then $\sigma \Psi_n = {-_n} {-_{n+1}} \Gamma_n({-_n} \sigma)$ can be expressed in terms of cubical $\omega$-groupoid operations.
  The operation $K_n \to K_{n+1}, \, x \mapsto x \Psi_n$ thus makes sense in any cubical $\omega$-groupoid $K_\bullet$, just like in Convention~\ref{conv Phi}.
\end{definition}

\begin{remark}
  We insist that our operations can be defined in arbitrary cubical $\omega$-groupoids as we will to apply them to arbitrary cofibrant operads in Sections~\ref{sec proof non formal 2}, \ref{sec sketch non formal 3}, and~\ref{sec proof non-formal all}.
\end{remark}

\subsection{Relative operads}
\label{sec relative operads}

We refer to \textcite{LodayVallette2012} and \textcite[Part I(a)]{Fresse2017} for background on operads.
Usually, a (symmetric) operad $\opP$ is indexed by integers ($\opP = \{ \opP(k) \}_{k \geq 0}$) and acted upon by symmetric groups $\Sigma_k \curvearrowright \opP(k)$.
We will mainly use an equivalent point of view where an operad is indexed by arbitrary finite sets ($\opP = \{ \opP(A) \}_{A \text{: finite set}}$) and acted upon by bijection of finite sets.
The two points of view are equivalent: given an operad indexed by finite sets, we simply define $\opP(k) \coloneqq \opP(\underline{k})$, where $\underline{k} = \{1, \dots, k\}$.
We will also work with colored operads, that is, operads where the inputs and the output of an operation are decorated by a color, and where composition is possible only if the colors match.
We will only considers colored operads of a special kind, called relative operads~\cite{Voronov1999} or Swiss-Cheese type operads~\cite{Willwacher2016}.
For such an operad $\opQ$, there are two colors, called respectively ``open'' (denoted by $\colo$) and ``closed'' (denoted by $\colc$).
An operation with a closed output may only have closed inputs, while an operad with an open input may have closed and open inputs.
The spaces of operations with a closed output $\opQ^{\colc} = \{ \opQ^{\colc}(A) \}_A$ thus form an ordinary operad.
The spaces of operations with an open output are denoted $\opQ^{\colo} = \{ \opQ^{\colo}(A) \}_A$.

\begin{convention}\label{conv:rel-op}
  Given a unicolored operad $\opP$, we will call a colored operad $\opQ$ such that $\opQ^{\colc} = \opP$ a \emph{relative} $\opP$-operad.
  Its components will be denoted $\opQ(A) \coloneqq \opQ^{\colo}(A)$ for a bicolored set $A$.
\end{convention}

The above convention follows the terminology of \textcite{Voronov1999}.
We will work heavily with a subset of operations in relative operads for which it will be useful to have coherent notation:

\begin{convention}\label{conv pm}
  Let us denote a set with two open-colored elements and its Cartesian powers by:
  \begin{align}
    {\pm}   & \coloneqq \{ +, - \}, \\
    {\pm}^l & \coloneqq \{ \star_1 \dots \star_l \mid \star_i \in \pm \}.
  \end{align}
  In other words, we view elements of $\pm^l$ as strings of signs, without parentheses around them or commas between them.
  We also write the unique element of $\pm^0$ as:
  \begin{equation}
    \colo \in \pm^0.
  \end{equation}
\end{convention}

\begin{remark}
  If an element $\star \in \pm$ appears in an arithmetic computation, then we take the convention that $+$ has the value $+1$ and $-$ has the value $-1$.
  For example, $\frac{\star+1}{2}$ is equal to $1$ if $\star = +$ and $0$ otherwise.
\end{remark}

\begin{convention}\label{conv Qkl}
  Let $\opQ$ be a relative operad.
  For integers $k,l \geq 0$, we will write
  \begin{equation}
    \opQ(k, {\pm}^l) \coloneqq \opQ(\{\colc_1, \dots, \colc_k\}, \pm^l),
  \end{equation}
  where $\colc_1, \dots, \colc_k$ are closed-colored and the elements of $\pm^l$ are open-colored.
  If $k = 1$, we will allow ourselves the notational convenience of writing $\colc \coloneqq \colc_1$.
\end{convention}

\subsubsection{Cubical operads}

Both functors $N$ and $\cK$ of the cubical Dold--Kan equivalence (Theorem~\ref{thm dold-kan}) and the counit $N\cK(X) \to X$ are (lax) monoidal.
The unit $K_\bullet \to \cK N(K_\bullet)$, however, is not monoidal.
The adjunction thus does not readily induce an equivalence of category between dg-$\K$-operads and operads in cubical $\K$-linear $\omega$-groupoids.
However, we can use a result of \textcite{SchwedeShipley2003} to upgrade the cubical Dold--Kan equivalence for operads:

\begin{corollary}[of {\cite[Theorem~6.5]{SchwedeShipley2003}}]\label{cor eqv cub operads}
  There exists a Quillen equivalence between relative dg-$\K$-operads and relative operads in cubical $\K$-linear $\omega$-groupoids, whose right adjoint is given by applying the normalized chains functor $N : \omega \mathsf{Gpd}_{\K} \to \Ch(\K)$ arity wise.
\end{corollary}

\begin{proof}
  Relative operads can be described as monoids in the category of pairs of symmetric/bisymmetric collections $(\opP(A), \opQ(B))_{A, B}$ (with $A$ ranging over finite sets and $B$ over finite, bicolored sets) with respect to the plethysm monoidal product.
  Just like in \cite[Section~4.2]{SchwedeShipley2003}, we can apply \cite[Theorem~3.12, part 3]{SchwedeShipley2003} to immediately obtain that the functor induced by $N$ from relative dg-$\K$-operads to relative operads in cubical $\K$-linear $\omega$-groupoids fits into a Quillen equivalence of models categories.
\end{proof}

We will require the following construction on certain elements of an operad of cubical chains over a topological operad:

\begin{definition}
  Let $\opP$ be a topological operad.
  For chains
  \begin{align*}
    \bar x
     & \in C_{m+1}\opP(A),
     & \bar y
     & \in \textstyle\prod_A C_{n^a+1}\opP(B^a),
  \end{align*}
  let
  \begin{align*}
    k & =m+\textstyle\sum_{A}n^a,
      & \Delta(\bar x(\bar y^a))
      & \in C_{k+1}\opP(\textstyle\coprod_A B^a)
  \end{align*}
  such that for $t\in[-1,1]^{\underline m\sqcup (\coprod_{A}\underline n^a)\sqcup\{k+1\}}$ we have
  $$
    \Delta(\bar x(\bar y^a))(t)
    =\bar x(\pr_{\underline m}t,t_{k+1})
    (\bar y^a(\pr_{\underline n^a}t,t_{k+1})).
  $$

  If we further have $\bar z\in \prod_{\coprod_A B^a}C_{p^{ab}+1}\opP(C^{ab})$ then let:
  $$
    \Delta(\bar x(\bar y^a(\bar z^{ab})))
    \coloneqq \Delta(\Delta(\bar x(\bar y^a))(\bar z^{ab}))
    =\Delta(\bar x(\Delta(\bar y^a(\bar z^{ab})))).
  $$
  This construction is extended to operads in cubical $\omega$-groupoids using the cubical analogue of the Alexander--Whitney map.
\end{definition}

\subsection{The little cubes and Swiss-Cheese operads}
\label{sec cubes operads}

Let us fix some dimension $n \geq 1$ throughout the paper.
In what follows, we are going to consider Voronov's Swiss-Cheese operad, denoted by $\SCv_{n+1}$~\cite{Voronov1999}.
Let us briefly recall its construction.
We first define the little $(n+1)$-cubes operad, denoted $\C_{n+1}$.

\begin{remark}
  Since the first coordinate will be distinguished to define the Swiss-Cheese operad, we will view elements of $[-1, 1]^{n+1}$ as being indexed by $[n] = \{0, \dots, n\}$, with the convention that $0$ is distinguished.
\end{remark}

\begin{definition}
  For two points $x, x' \in [-1, 1]^{n+1}$ satisfying $x_i < x'_i$ for all $i \in [n]$, denote the cuboid with lower-left corner $x$ and upper-right corner $x'$ by:
  \[
    C(x; x') \coloneqq \{ y \in [-1, 1]^{n+1} \mid \forall i \in [n], \, x_i \leq y_i \leq x'_i \}.
  \]
\end{definition}

\begin{definition}
  Let $A$ be a finite set.
  The component of the \emph{little $(n+1)$-cubes operad} indexed by $A$, denoted $\C_{n+1}(A)$, is the collection of $A$-tuples of pairs of points $x = \{x(a); x'(a)\}_{a \in A}$ in $([-1, 1]^{n+1})^{A \sqcup A}$, satisfying the following conditions:
  \begin{itemize}
    \item For all $a \in A$ and $i \in [n]$, we have $x(a)_i < x'(a)_i$;
    \item For all $a \neq b \in A$, the cubes $C(x(a); x'(a))$ and $C(x(b); x'(b))$ have disjoint interiors.
  \end{itemize}
  We can view a pair $(x(a), x'(a))$ as a rectilinear embedding $[-1, 1]^{n+1} \hookrightarrow [-1, 1]^{n+1}$ defined by the inclusion $C(x(a), x'(a)) \subseteq [-1, 1]^{n+1}$.
  Operadic composition is defined by composition of embeddings (see Figure~\ref{fig op compos c2}).
\end{definition}

\begin{figure}[htbp]
  \centering
  \raisebox{-.5\height}{\includegraphics[width=.25\linewidth]{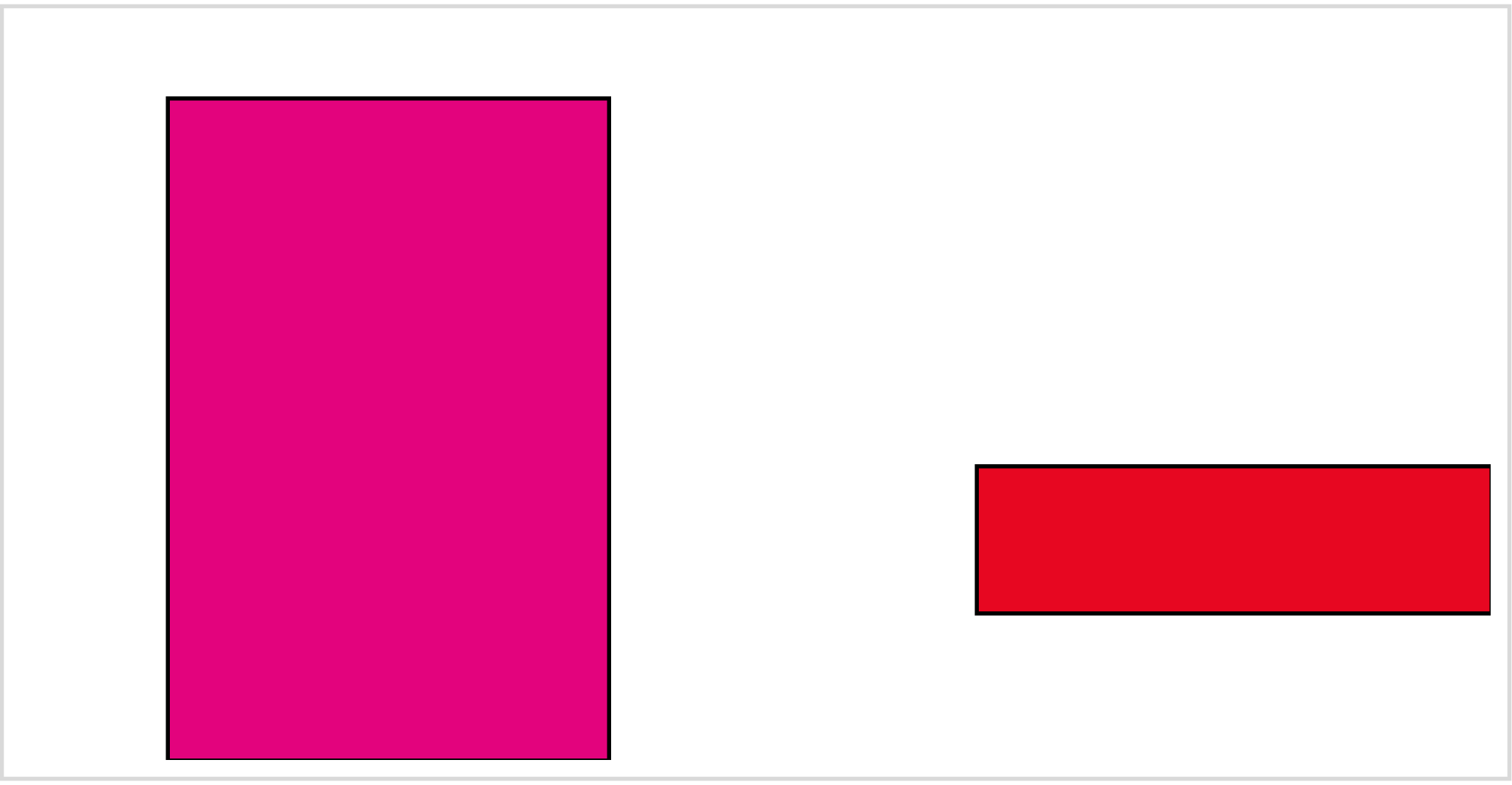}}
  $\circ_2$
  \raisebox{-.5\height}{\includegraphics[width=.25\linewidth]{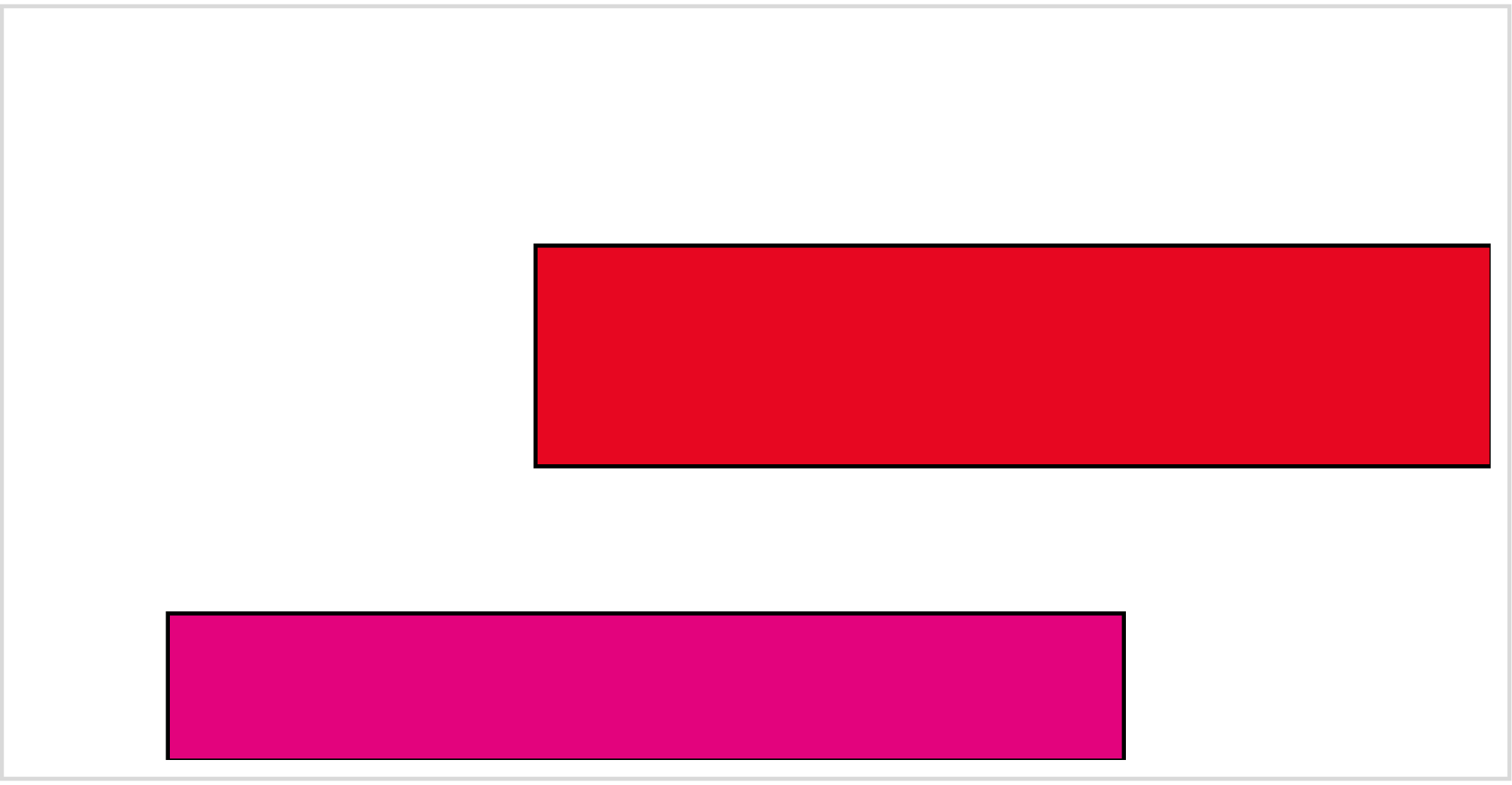}}
  =
  \raisebox{-.5\height}{\includegraphics[width=.25\linewidth]{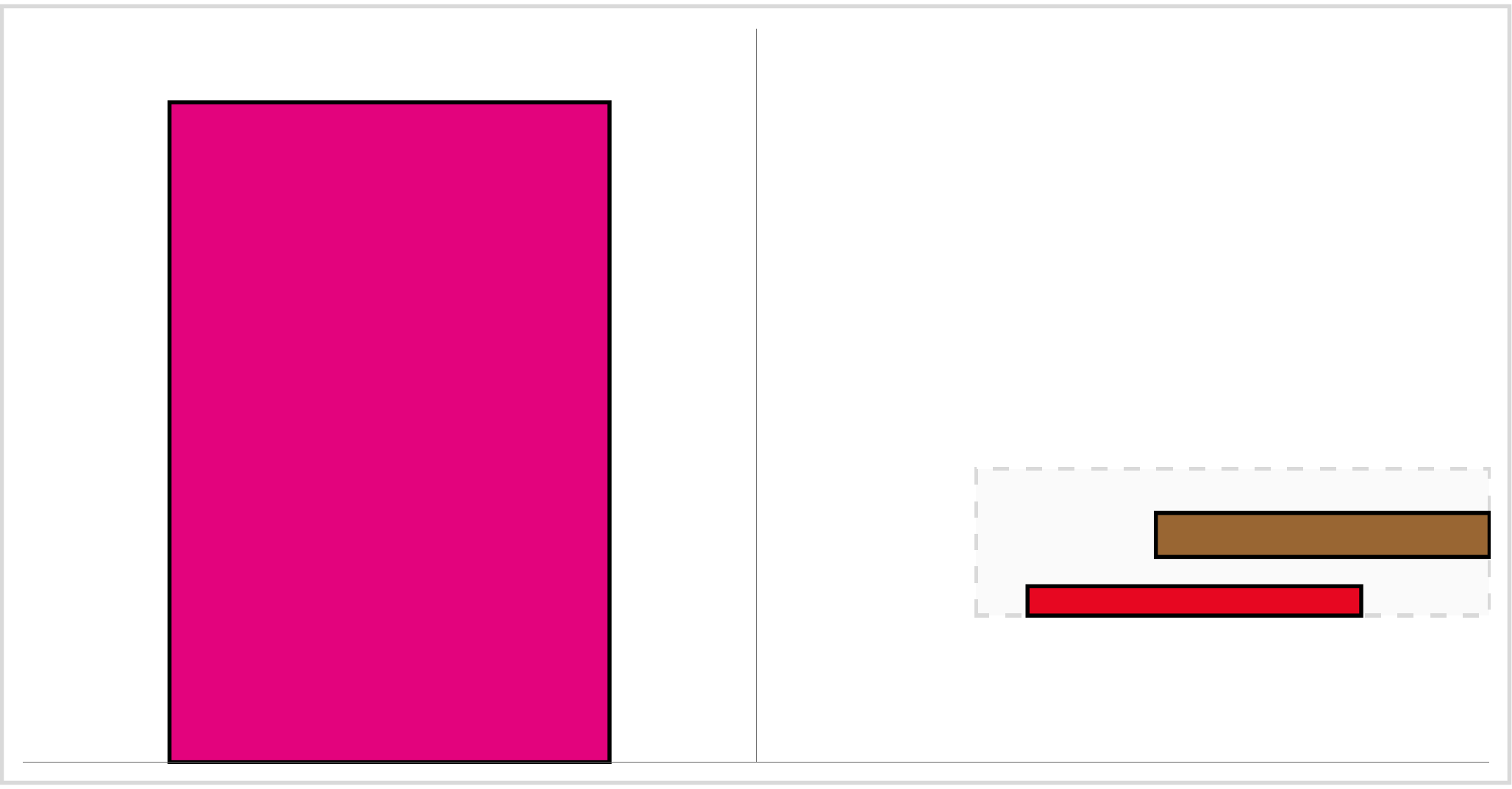}}
  \caption{Operadic composition in $\C_2$.}
  \label{fig op compos c2}
\end{figure}

We consider the version of the Swiss-Cheese operad defined by \textcite{Voronov1999} as follows.

\begin{definition}\label{def:scv}
  Let $n \geq 2$ be an integer and let $A$ be a finite set whose elements are decorated either by the open or the closed color.
  The \emph{Swiss-Cheese} operad $\SCv_{n+1}$ is the relative operad defined by (where $A$ is a bicolored set and $\bar{A}$ is its underlying set):
  \begin{itemize}
    \item The space of operations with a closed output, $(\SCv_{n+1})^{\colc}(A)$, is equal to $\C_{n+1}(A)$ if all the inputs in $A$ are closed, and it is empty otherwise.
    \item The space of operations with an open output, $(\SCv_{n+1})^{\colo}(A)$, is empty if $A$ only contains closed inputs.
    \item Otherwise, if $A$ contains at least one open input, then $(\SCv_{n+1})^{\colo}(A)$ is given by the elements $\{x(a); x'(a)\} \in \C_n(\bar{A})$ such that for each $a \in A$ the ``lower-left corner'' $x(a)$ belongs to the ambient $n$-cube $[0, 1] \times [-1, 1]^n$. If $a \in A$ is open then $x(a)$ belongs to the face $\{0\} \times [-1, 1]^{n}$.
  \end{itemize}
  The composition in $\SCv_{n+1}$ is induced by that of $\C_{n+1}$.
\end{definition}

Following Convention~\ref{conv:rel-op}, given a bicolored set $A$ we will simply write $\C_{n+1}(A)$ instead of $(\SCv_{n+1})^{\colc}(A)$ and $\SCv_{n+1}(A)$ instead of $(\SCv_{n+1})^{\colo}(A)$ when no confusion can arise.

\begin{remark}\label{rmk:diff-voronov}
  We are forbidding operations with no open inputs, like in the original paper of \textcite{Voronov1999}.
  There is a different version of the Swiss-Cheese operad, denoted $\SC_{n+1}$, that includes these operations.
  This operad has been studied elsewhere (see the next section) and is now usually called the Swiss-Cheese operad.
\end{remark}

\begin{convention}
  For notational convenience, we will generally denote an element $x \in \SC_{n+1}(\{\colc_1, \dots, \colc_k\}, \{\colo_1, \dots, \colo_l\})$ under the form:
  \begin{equation}
    x =
    \begin{cases}
      [x_0(\colc_1), x'_0(\colc_1)] \times \dots \times [x_n(\colc_1), x'_n(\colc_1)]
       & (\colc_1)  \\
      \dots         \\
      [x_0(\colc_k), x'_0(\colc_k)] \times \dots \times [x_n(\colc_k), x'_n(\colc_k)]
       & (\colc_k)  \\
      [x_0(\colo_1), x'_0(\colo_1)] \times \dots \times [x_n(\colo_1), x'_n(\colo_1)]
       & (\colo_1)  \\
      \dots         \\
      [x_0(\colo_l), x'_0(\colo_l)] \times \dots \times [x_n(\colo_l), x'_n(\colo_l)]
       & (\colo_l),
    \end{cases}
  \end{equation}
  where $x_i(\dots), x'_i(\dots) \in [-1, 1]$ are the endpoints of the intervals that make up the cubes $x(\colc_1), \dots, x(\colc_k), x(\colo_1), \dots, x(\colo_l)$.
\end{convention}

\subsection{Turning squares into cubes}
\label{sec squares into cubes}

There are several possible ways of embeddings the little squares operad $\C_2$ inside the little cubes operad $\C_3$.
The classical inclusion $\inc[0]{2} : \C_2 \to \C_3$ sends a configuration of squares to a configuration of cubes of height $2$.
In other words, given a configuration $x = \bigl\{[x_0(a), x'_0(a)] \times [x_1(a), x'_1(a)] \bigr\}_{a \in A} \in \C_2(A)$, we apply the following map to each element of $x$:
\begin{equation}\label{ex:c0}
  \inc[0]{2} \bigl( [x_0(a), x'_0(a)] \times [x_1(a), x'_1(a)] \bigr) \coloneqq [-1, 1] \times [x_0(a), x'_0(a)] \times [x_1(a), x'_1(a)].
\end{equation}

\begin{figure}[htbp]
  \centering
  \subcaptionbox{$x \in \C_2(2)$}[.24\linewidth]{\includegraphics[width=\linewidth]{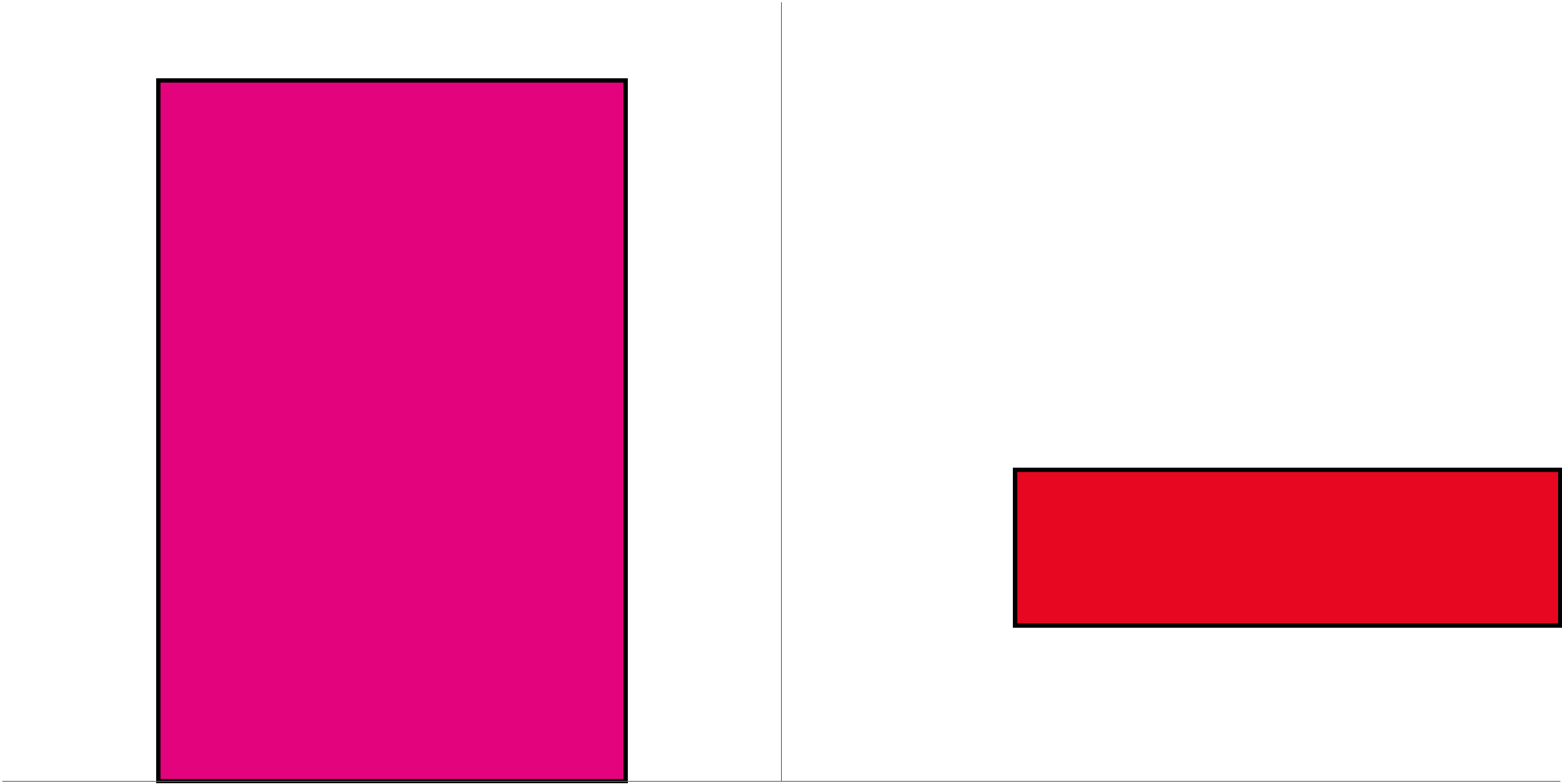}}
  \subcaptionbox{$\inc[0]{2} x \in \C_3(2)$}[.24\linewidth]{\includegraphics[width=\linewidth]{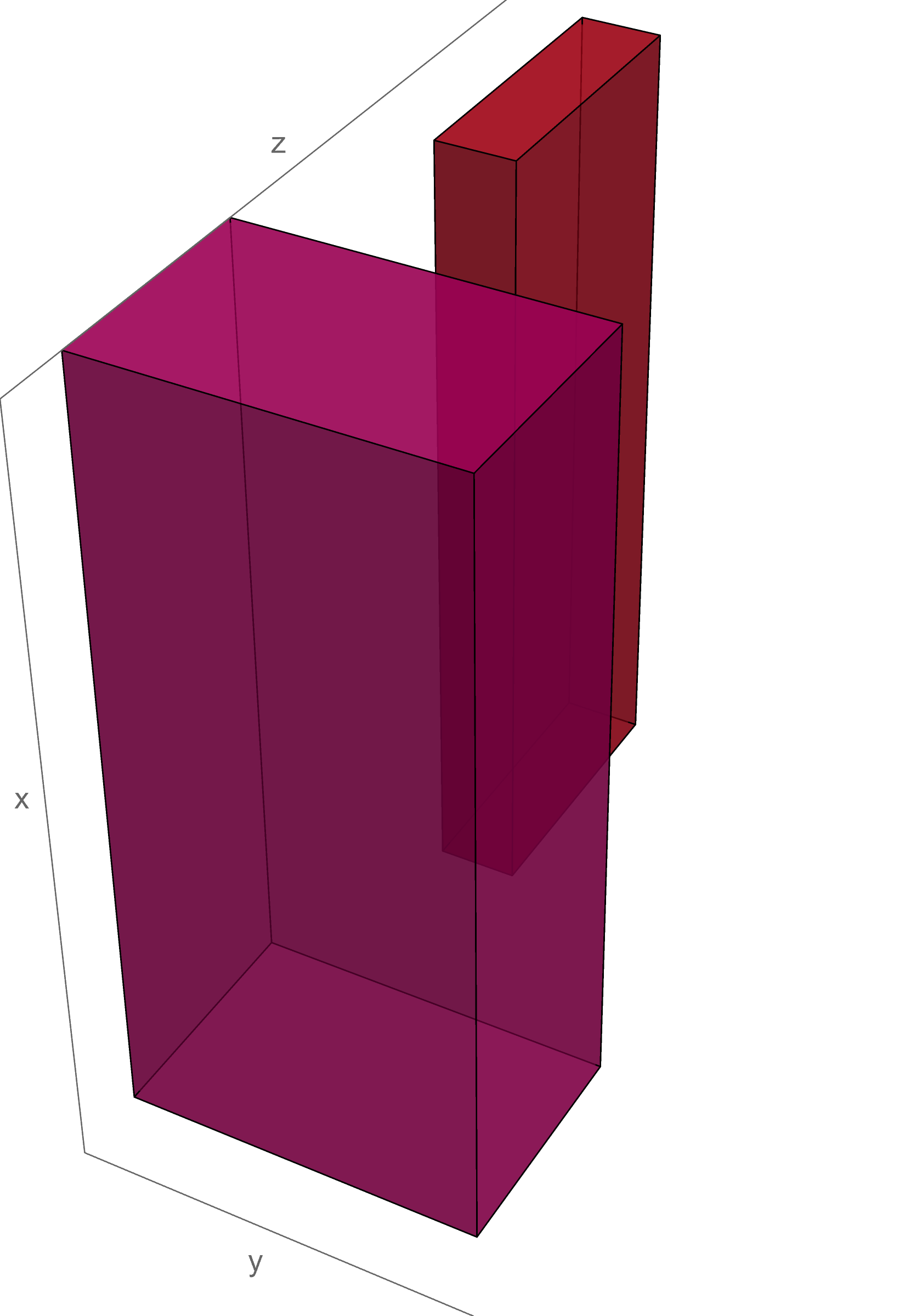}}
  \subcaptionbox{$\inc[1]{2} x \in \C_3(2)$}[.24\linewidth]{\includegraphics[width=\linewidth]{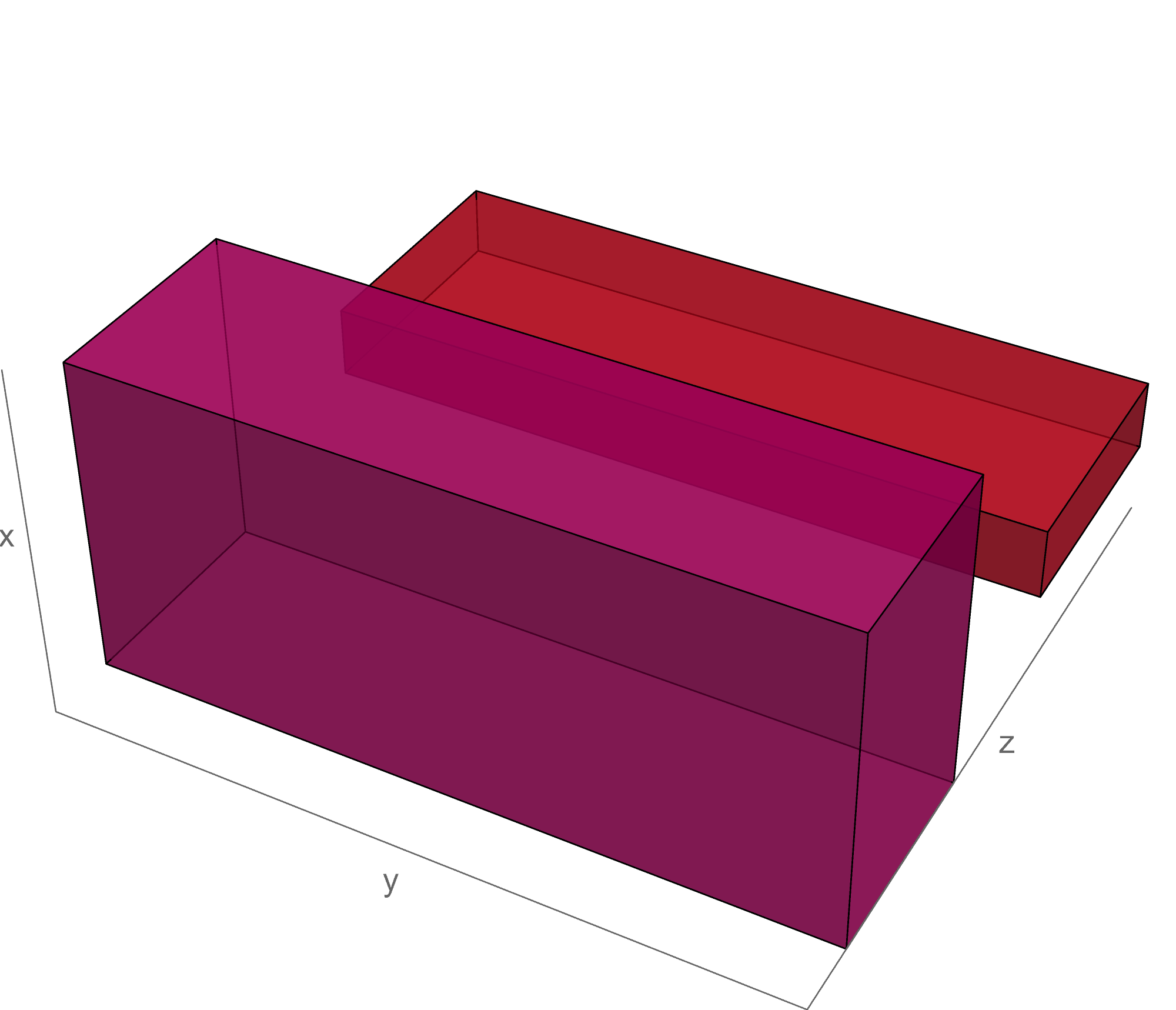}}
  \subcaptionbox{$\inc[2]{2} x \in \C_3(2)$}[.24\linewidth]{\includegraphics[width=\linewidth]{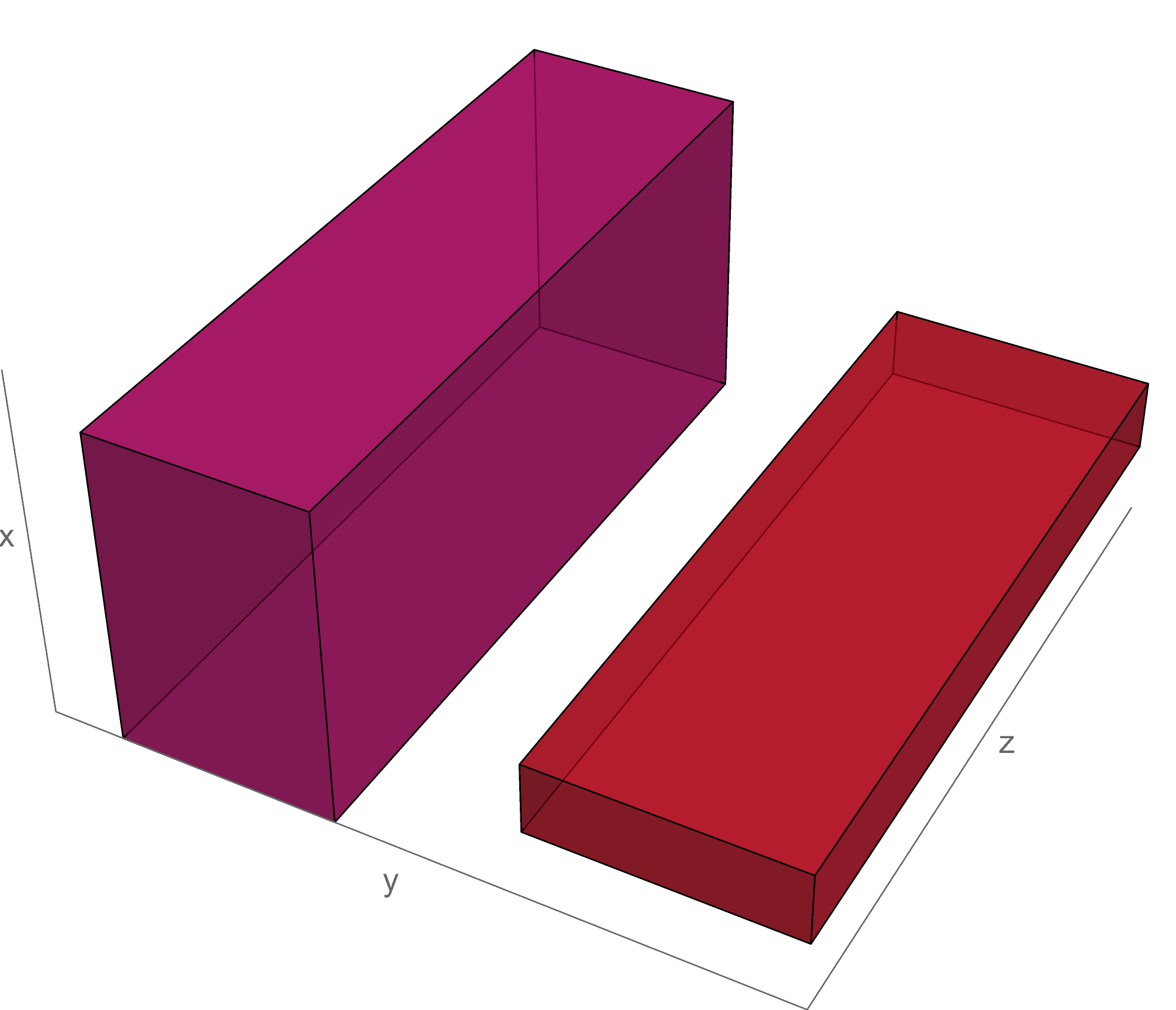}}
  \caption{The three inclusions $\C_2 \to \C_3$.}
  \label{fig:incl-c2-c3}
\end{figure}

This procedure is illustrated in Figure~\ref{fig:incl-c2-c3}.
However, when dealing with (cube-based) Swiss-Cheese operads, there are two other embeddings that are meaningful to consider.
These two embeddings are defined similarly to $\inc[0]{2}$ above, except that their images consist in configuration of cubes of width (resp. depth) $2$, rather than height $2$.
In order to have coherent notation with the next section, let us introduce the following definition.

\begin{definition}\label{def inc2}
  Let $x = \{[x_0(a), x'_0(a)] \times [x_1(a), x'_1(a)]\}_{a \in A} \in \C_2(A)$ be a configuration of squares for some finite set $A$.
  Define two elements $\inc[1]{2}x, \inc[2]{2}x\in \C_3(A)$ by applying these two maps to each cube in $x$ (see Figure~\ref{fig:incl-c2-c3}):
  \begin{align}
    \inc[1]{2} \bigl( [x_0(a), x'_0(a)] \times [x_1(a), x'_1(a)] \bigr)
     & \coloneqq [x_0(a), x'_0(a)] \times [-1, 1] \times [x_1(a), x'_1(a)].
    \\
    \inc[2]{2} \bigl( [x_0(a), x'_0(a)] \times [x_1(a), x'_1(a)] \bigr)
     & \coloneqq [x_0(a), x'_0(a)] \times [x_1(a), x'_1(a)] \times [-1, 1],
  \end{align}
\end{definition}

In the previous definition, the superscripts refer to the dimension of the cubes in the target operad, and the subscripts refer to the positions at which the square is widened.
More generally, let us define the following inclusions.
Since we will never need the case $0 \in S$, we will omit it to avoid certain complications.

\begin{definition}\label{def inc}
  Let $1 \leq k \leq n$ be integers and $S = \{ s_1 < \dots < s_{k} \} \subseteq \underline{n}$. Consider the complement $S^c = \underline{n} \setminus S = \{ u_1 < \dots < u_{n-k} \}$.
  Let $x = \bigl\{ \prod_i [x_i(a), x'_i(a)] \bigr\}_{a \in A} \in \C_{n-k+1}(A)$ be a configuration of $(n-k+1)$-cubes for some finite set $A$.
  Let us define a new element:
  \begin{equation}
    \inc[S]{n} x = \inc[\{s_1, \dots, s_k\}]{n} x \in \C_{n+1}(A)
  \end{equation}
  by, for all $a \in A$,
  \begin{equation}
    (\inc[S]{n} x)(a) \coloneqq [x_0(a), x'_0(a)] \times \prod_{j=1}^{n} [y_j(a), y'_j(a)],
  \end{equation}
  where:
  \begin{align}
    y_j(a)
     & =
    \begin{cases}
      x_i(a), & \text{if } \exists i \text{ s.t. } u_i = j; \\
      -1,     & \text{if } \exists i \text{ s.t. } s_i = j.
    \end{cases}
     &
    y'_j(a)
     & =
    \begin{cases}
      x'_i(a), & \text{if } \exists i \text{ s.t. } u_i = j; \\
      1,       & \text{if } \exists i \text{ s.t. } s_i = j.
    \end{cases}
  \end{align}
\end{definition}

\begin{remark}
  Recall the maps $\inc[i]{2}$ defined at the beginning of the section.
  We obviously have $\inc[i]{2} = \inc[\{i\}]{2}$ for $i \in \{1,2\}$.
\end{remark}

The proof of the following proposition is straightforward.

\begin{proposition}
  The map $\inc[S]{n} : \C_{n-|S|+1} \to \C_{n+1}$, for $S \subseteq \underline{n}$, is an inclusion of operads and restricts to an inclusion of operads $\SCv_{n-|S|+1} \to \SCv_{n+1}$.
\end{proposition}

We will use the following convention about composition of operations that will prove useful in Sections~\ref{sec non-formal 3} and~\ref{sec non-formal all}.
Its main interest is that it avoids the search for the appropriate permutation of inputs that yield the correct operations.

\begin{convention}\label{conv compos pm}
  Let $n \geq 0$ be an integer, $S, T \subseteq \underline{n}$ be two \emph{disjoint} sets, with $S = \{s_1<\dots<s_k\}$ and $T=\{t_1<\dots<t_l\}$.
  Suppose that we are given an operation $x \in \SCv_{|S|+1}(C, \pm^S)$ for some set of closed inputs $C$, and operations $y_{\star} \in \SCv_{|T|+1}(C_{\star}, \pm^T)$ for all $\star \in \pm^{|S|}$.
  Then we will view the operation $\inc[\underline{n} \setminus S]{n} x \bigcirc_{\star \in \pm^{|S|}} \inc[\underline{n} \setminus T]{n} y_{\star}$ as an element of the following space:
  \begin{equation}
    \inc[\underline{n} \setminus S]{n} x \underset{\star \in \pm^{|S|}}{\bigcirc} \inc[\underline{n} \setminus T]{n} y_{\star} \in
    \SCv_{n+1}\bigl( C \sqcup \bigsqcup_{\star \in \pm^S} C_\star, \; \pm^{S \sqcup T} \bigr),
  \end{equation}
  by declaring that the input indexed by $\star' = \star'_{t_1} \dots \star'_{t_l}$ in $y_\star$, where $\star = \star_1 \dots \star_k$, becomes indexed by the concatenation $\star \star'$ reordered such that the indices $s_i, t_j$ appear in their natural order in $S \sqcup T \subseteq \underline{n}$.
\end{convention}

We may think of the decorations as ``bubbling up'' and being distributed to the indices of the open inputs $y_{*}$.
While this convention may appear opaque at first, it greatly simplifies formulas.
Let us illustrate it with an example.
Since closed inputs are not relevant in the discussion, we will consider operation without any closed inputs.

\begin{example}
  Let $n = 5$, $S = \{1, 4\}$, and $T = \{2\}$.
  Let $x \in \SCv_3(\emptyset, \pm^{\{1,4\}})$ and, for all $\star \in \pm^{\{1,4\}}$, let $y_\star \in \SCv_2(\emptyset, \pm^{\{2\}})$.
  Then we consider the colors of the composition:
  \begin{equation*}
    \inc[\{2,3,5\}]{5}x \underset{\star \in \pm^S}{\bigcirc} \inc[\{1,3,4,5\}]{5} y_\star
    \in \SCv_6(\emptyset, \pm^{\{1,2,4\}})
  \end{equation*}
  to be defined as in the following illustration.
  We circle the signs coming from $x$ while we write those coming from $y$ inside squares to distinguish them.
  \tikzset{optree/.style={grow'=up, every child node/.style={font=\small}, sibling distance=12mm, level distance=8mm}}
  \forestset{
    operad/.style={
        font=\small,
        for tree={grow'=90, l sep=.5cm, s sep = .25cm}}
  }

  \begin{multline*}
    \begin{forest}{operad}
      [{$\inc[\{2,3,5\}]{5} x$}, baseline
          [{$\color{red}{\oplus}{\oplus}$}]
          [{$\color{red}{\oplus}{\ominus}$}]
          [{$\color{red}{\ominus}{\oplus}$}]
          [{$\color{red}{\ominus}{\ominus}$}]
      ]
    \end{forest}
    \circ
    \Bigl(
    \begin{forest}{operad}
      [{$\inc[\{1,3,4,5\}]{5} y_{++}$}, baseline
          [{$\color{blue}{\boxplus}$}]
          [{$\color{blue}{\boxminus}$}]
      ]
    \end{forest}
    ,
    \begin{forest}{operad}
      [{$\inc[\{1,3,4,5\}]{5} y_{+-}$}, baseline
          [{$\color{blue}{\boxplus}$}]
          [{$\color{blue}{\boxminus}$}]
      ]
    \end{forest}
    ,
    \begin{forest}{operad}
      [{$\inc[\{1,3,4,5\}]{5} y_{-+}$}, baseline
          [{$\color{blue}{\boxplus}$}]
          [{$\color{blue}{\boxminus}$}]
      ]
    \end{forest}
    ,
    \begin{forest}{operad}
      [{$\inc[\{1,3,4,5\}]{5} y_{--}$}, baseline
          [{$\color{blue}{\boxplus}$}]
          [{$\color{blue}{\boxminus}$}]
      ]
    \end{forest}
    \Bigr)
    \\
    =
    \begin{forest}{operad}
      [{$\inc[\{2,3,5\}]{5} x$}, baseline
          [{$\inc[\{1,3,4,5\}]{5} y_{++}$}
              [{$\color{red}{\oplus}\color{blue}{\boxplus}\color{red}{\oplus}$}]
              [{$\color{red}{\oplus}\color{blue}{\boxminus}\color{red}{\oplus}$}]
          ]
          [{$\inc[\{1,3,4,5\}]{5} y_{+-}$}
              [{$\color{red}{\oplus}\color{blue}{\boxplus}\color{red}{\ominus}$}]
              [{$\color{red}{\oplus}\color{blue}{\boxminus}\color{red}{\ominus}$}]
          ]
          [{$\inc[\{1,3,4,5\}]{5} y_{-+}$}
              [{$\color{red}{\ominus}\color{blue}{\boxplus}\color{red}{\oplus}$}]
              [{$\color{red}{\ominus}\color{blue}{\boxminus}\color{red}{\oplus}$}]
          ]
          [{$\inc[\{1,3,4,5\}]{5} y_{--}$}
              [{$\color{red}{\ominus}\color{blue}{\boxplus}\color{red}{\ominus}$}]
              [{$\color{red}{\ominus}\color{blue}{\boxminus}\color{red}{\ominus}$}]
          ]
      ]
    \end{forest}
    .
  \end{multline*}
\end{example}

\subsection{Formality and non-formality}
\label{sec:formality nonformality}

Our aim in this paper is to establish the non-formality of Voronov's Swiss-Cheese operad $\SCv_{n+1}$.
Let us briefly recall this notion and earlier results about the (non-)formality of the little cubes and Swiss-Cheese operads.

First, let us note that the functors $C_*$ and $H_*$ are lax-monoidal using the Künneth maps.
It thus follows that given a (possibly colored) topological operad $\opP$, the collections $C_* \opP = \{ C_*(\opP(A)) \}$ and $H_* \opP = \{ H_*(\opP(A)) \}$ form operads in chain complexes.

\begin{definition}\label{def:formal}
  A topological operad $\opP$ is \emph{(stably) formal} (over $\K$) if $C_* \opP$ and $H_* \opP$ are quasi-isomorphic, i.e., if there exists a zigzag of morphisms of dg-operads which are quasi-isomorphisms in every arity:
  $$C_* \opP \xleftarrow{\sim} \opQ_1 \xrightarrow{\sim} \dots \xleftarrow{\sim} \opQ_k \xrightarrow{\sim} H_* \opP.$$
\end{definition}

\begin{remark}
  The terminology comes from \cite[Theorem 1.1]{LambrechtsVolic2014}.
  If $\K = \Q$, then there is a stronger notion of formality that require the Hopf cooperad given by the cohomology of $\opP$ to be a rational model of $\opP$ (see e.g. \cite[Part II(b)]{Fresse2017a} for background).
  Strong formality implies stable formality.
  In the sequel, we are going to drop the adverb ``stably'' as we are only ever going to talk about stable formality.
\end{remark}

\begin{remark}
  Using arguments from model category theory, an operad $\opP$ is formal if and only if there exists a span $C_* \opP \xleftarrow{\sim} \opQ \xrightarrow{\sim} H_* \opP$ (see e.g.~\cite{Hirschhorn2003}).
\end{remark}

Several (non-)formality results are known about the little cubes operads and its cousins:
\begin{itemize}[nosep]
  \item If $n=1$, then the question of formality becomes much simpler.
        The little intervals operad $\C_1$ is formal over any ring.
        As a topological operad, it is weakly equivalent to the (discrete) operad of associative algebras $\Ass = \{ \Sigma_r \}_{r \geq 0}$, with a direct map $\C_1 \to \Ass = \pi_0 \C_1$.
  \item The Swiss-Cheese operad $\SC_1$ is thus also formal over any ring.
        There is again a direct map $\SC_1 \to \pi_0 \SC_1$ which is a weak homotopy equivalence.
  \item The little cubes operad $\C_n$ is formal over $\Q$~\cite{Kontsevich1999,Tamarkin2003,LambrechtsVolic2014} for every $n$.
        It is however not formal over a field of positive characteristic, simply because $C_*\C_n(p)$ is not formal as a $\Sigma_p$-dg-module.
        For $n = 2$, it is also known to not be formal even if one forgets the action of the symmetric group (\textcite{Salvatore2018}).
  \item There is an obvious inclusion of operads $\C_n \hookrightarrow \C_{n+k}$ for $k \geq 0$.
        Definition~\ref{def:formal} easily generalizes to the definition of a formal morphism of operads (see e.g. \cite[Definition~1.3]{LambrechtsVolic2014}).
        The inclusion $\C_n \hookrightarrow \C_{n+k}$ is only formal over $\Q$ if $k \neq 1$, see \cite[Theorem~1.4]{LambrechtsVolic2014} and~\cite{TurchinWillwacher2018,FresseWillwacher2015}.
        This is related to the formality of the Swiss-Cheese operad (see the next point).
  \item The larger version of the Swiss-Cheese operad $\SC_{n+1}$ that allows operation without open inputs (see Remark~\ref{rmk:diff-voronov}) is already known not to be formal for $n \geq 2$.
        This was shown by \textcite{Livernet2015} in characteristic different from $2$, and it was shown by \textcite{Willwacher2017a} over $\Q$ (by proving that it is equivalent to the (non-)formality for $\C_{n-1} \to \C_n$).
        However, both proofs make a crucial use of the operations without open inputs, so they do not settle the formality of $\SCv_{n+1}$.
        Our proof is inspired by Livernet's proof for $n = 2$.
\end{itemize}

\section{Non-formality in dimension 2}
\label{sec non formal 2}

In this section, we record the proof of the non-formality of Voronov's Swiss-Cheese operad $\SCv_2$ that can be found in the second-named author's thesis~\cite{Vieira2018a}.

\subsection{Definition of the basic paths}
\label{sec:def-basic-paths}

Let us start by defining some basic elements of the chain complex of $\SCv_2$.
The elements will all be paths, that can then be viewed as 1-chains.
Recall Convention~\ref{conv Qkl} about the notation for relative operads, that we will apply to $\SCv_2$.

\begin{definition}\label{def pro1}\label{def act1}
  The product $\pro{1}$ and the action $\act{1}$ are defined by:
  \[
    \pro{1} =
    \begin{cases}
      [0, 1] \times [-1, 0] & (-) \\
      [0, 1] \times [0, 1]  & (+)
    \end{cases}
    =
    \raisebox{-.5\height}{\includegraphics[width=12em]{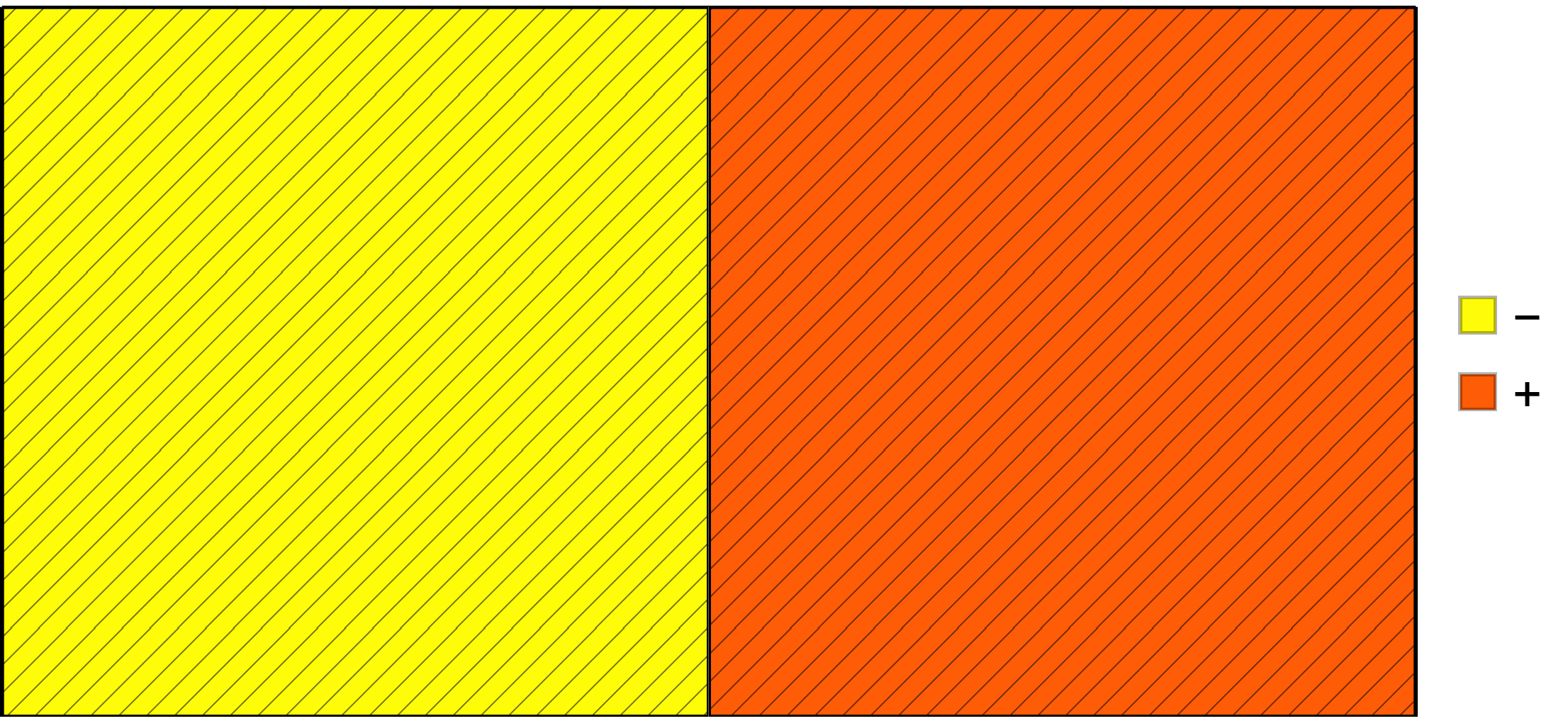}}
    \in C_0(\SCv_2(0, {\pm})),
  \]
  \[
    \act{1}
    =
    \begin{cases}
      [1/2, 1] \times [-1, 1] & (\colc) \\
      [0, 1/2] \times [-1, 1] & (\colo)
    \end{cases}
    =
    \raisebox{-.5\height}{\includegraphics[width=12em]{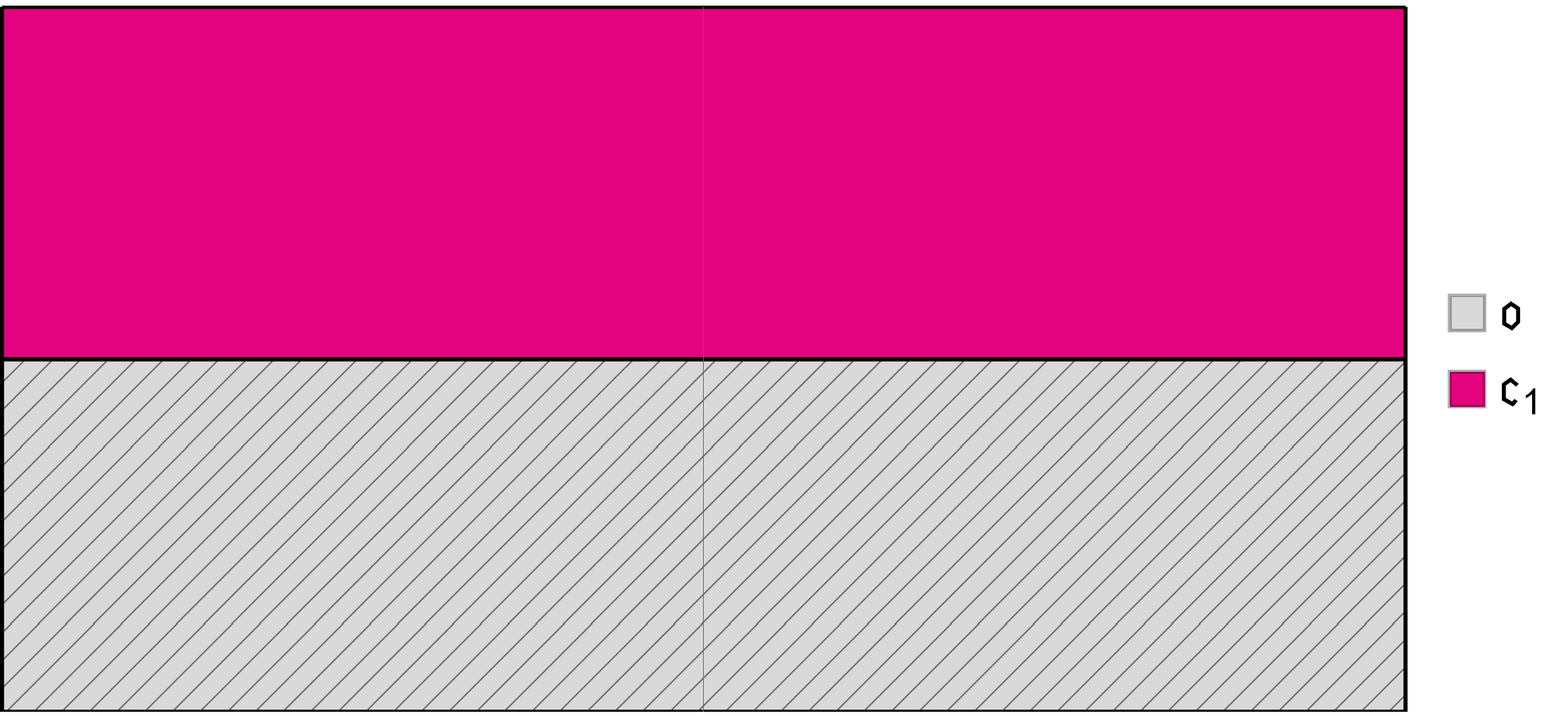}}
    \in C_0(\SCv_2(1, {\pm}^0)).
  \]
\end{definition}

Recall that we use cubical chains (see Section~\ref{sec background}) and that 1-chains of a space $X$ are linear combinations of paths $[-1, 1] \to X$.
We will now define two paths in $\SCv_2(1, \pm)$ that will be essential to produce the obstruction to formality.

\begin{definition}\label{def push1}
  The paths $\push[+]{1}$ and $\push[-]{1}$ in $C_1(\SCv_2(1, \pm))$ are defined by (for $-1 \leq t \leq 1$):
  \begin{align*}
    \push[+]{1}(t)
     & \coloneqq
    \begin{cases}
      [1/2, 1] \times [\min(t, 0), 1]                         & (\colc) \\
      \bigl[ 0, \max(\frac{1+t}{2}, 1/2) \bigr] \times [0, 1] & (-)     \\
      [0, 1/2] \times [-1, 0]                                 & (+)
    \end{cases}
    \\
    \push[+]{1}
     & =
    \raisebox{-.5\height}{\includegraphics[height=2cm]{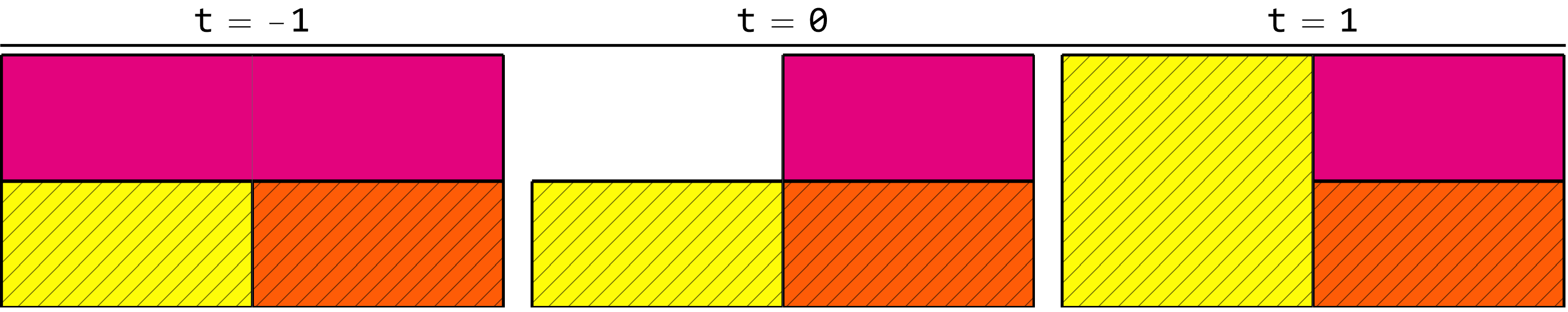}}
    \in
    C_1(\SCv_2(1, {\pm})),
  \end{align*}

  \begin{align*}
    \push[-]{1}(t)
     & \coloneqq
    \begin{cases}
      [1/2, 1] \times [-1, \max(0, -t)]                       & (\colc) \\
      [0, 1/2] \times [-1, 0]                                 & (-)     \\
      \bigl[ 0, \max(\frac{1+t}{2}, 1/2) \bigr] \times [0, 1] & (+)
    \end{cases}
    \\
    \push[-]{1}
     &
    =
    \raisebox{-.5\height}{\includegraphics[height=2cm]{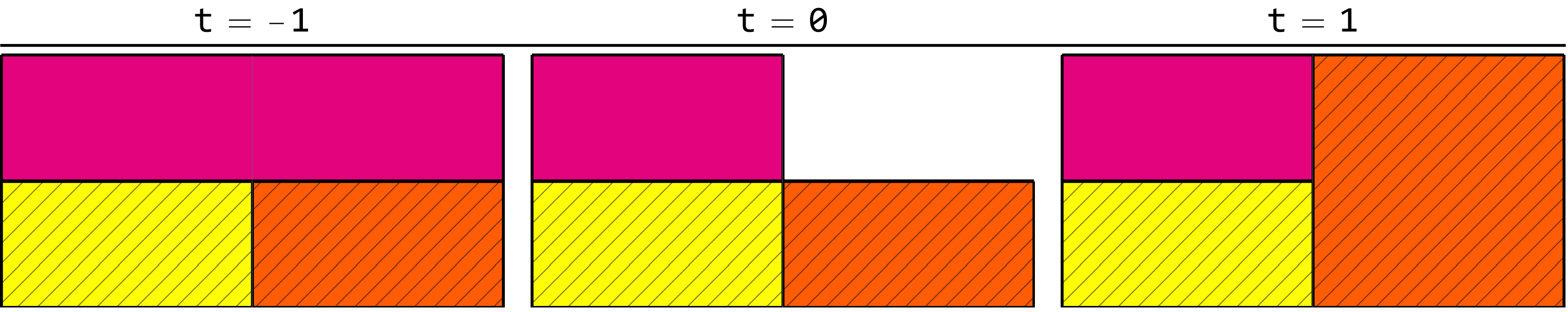}}
    \in
    C_1(\SCv_2(1, {\pm})).
  \end{align*}
\end{definition}

\begin{remark}\label{rmk beta 01}
  For $\star\in\pm$ we have that:
  \begin{align*}
    d^-_1\push[\star]{1}
     & =\act{1}\circ_{\colo}\pro{1}
     &
    d^+_1\push[\star]{1}
     & =\pro{1}\circ_\star\act{1}
  \end{align*}
\end{remark}

\begin{remark}
  In Remark~\ref{rmk shortcut}, we will give a more condensed definition of $\push[+]{1}$ and $\push[-]{1}$ that lends itself to generalization in higher dimension.
\end{remark}

\begin{definition}\label{def spec1}
  The 1-chain $\spec[+]{1} \in C_1(\SCv_2(2, \pm))$ is defined by:
  \[
    \spec[+]{1} \coloneqq
    - \push[+]{1} \circ_{-} \act{1}
    - \act{1} \circ_{\colo}\push[-]{1}
    + \act{1} \circ_{\colo} \push[+]{1}
    + \push[-]{1} \circ_{+} \act{1}   .
  \]
  We further set $\spec{1} \coloneqq \spec[+]{1}\cdot(1+(\colc_1\colc_2))\in C_1(\SCv_2(2, \pm))$.
\end{definition}

Note that despite being defined as a chain, $\spec{1}$ is the concatenation of eight paths with compatible endpoints.

\begin{equation*}
  \begin{tikzcd}[column sep = 3cm]
    {} \arrow[r, "{-\act{1} \circ_{\colo} \push[-]{1}}" ]
    & {}
    \arrow[r, "{\act{1}\circ_{\colo}\push[+]{1}}"]
    & {} \arrow[d, "{\push[-]{1}\circ_+\act{1}}"]
    \\
    {} \arrow[u, "{-\push[+]{1}\circ_-\act{1}}" ]
    & {}
    & {} \arrow[d, "{-\push[+]{1}\circ_-\act{1}\cdot(\colc_1 \colc_2)}"] \\
    {} \arrow[u, "{\push[-]{1}\circ_+\act{1}\cdot(\colc_1 \colc_2)}"]
    & {}
    \arrow[l, "{\act{1}\circ_{\colo}\push[+]{1}\cdot(\colc_1 \colc_2)}"]
    & {} \arrow[l, "{-\act{1}\circ_{\colo}\push[-]{1}\cdot(\colc_1 \colc_2)}"]
  \end{tikzcd}
\end{equation*}

Graphically, the path $\spec[+]{1}$ can be represented by Figure~\ref{fig:eta}.

\begin{figure}[htbp]
  \centering
  \includegraphics[width=\linewidth]{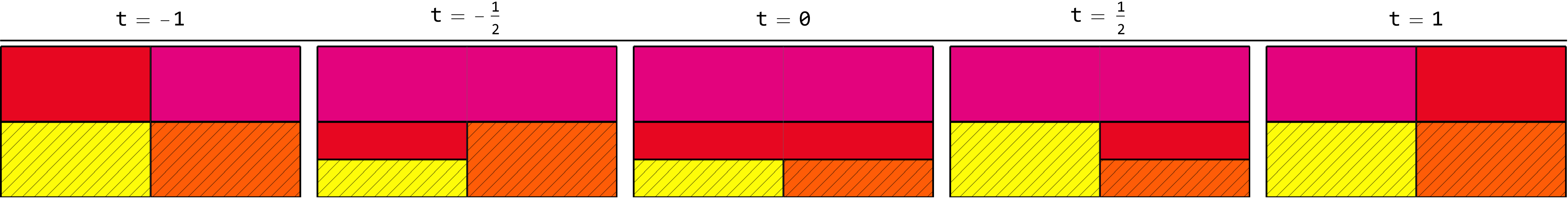}
  \caption{The 1-chain $\spec[+]{1} \in C_1 \SCv_2(2, \pm)$.}
  \label{fig:eta}
\end{figure}

\begin{lemma}
  The path $\spec{1}$ is closed, and thus $ d \spec{1} = 0$ in the chain complex of $\SCv_2(2, \pm)$.
\end{lemma}
\begin{proof}
  This follows directly from Remark~\ref{rmk beta 01}.
\end{proof}

It is not hard to see that the above closed path is related to the ``loop'' $\loo{1}$ in $\C_2(2)$:
\begin{definition}\label{def loo1}
  The path $\loo[+]{1} \in C_1 \C_2(2)$ is defined as
  \begin{multline*}
    \loo[+]{1}(t)
    \coloneqq
    \Biggl\{
    \Bigl[\frac{1-\lvert 2t\rvert}{4},\frac{3-\lvert 2t\rvert}{4}\Bigr]
    \times
    \Bigl[\frac{-1-2t}{4},\frac{1-2t}{4}\Bigr]
    ,
    \\
    \Bigl[\frac{-3+\lvert 2t\rvert}{4},\frac{-1+\lvert 2t\rvert}{4}\Bigr]
    \times
    \Bigl[\frac{-1+2t}{4},\frac{1+2t}{4}\Bigr]
    \Biggr\}
  \end{multline*}
  We set $\loo{1} \coloneqq \loo[+]{1} \cdot (1+(\colc_1\colc_2))$.
\end{definition}

\begin{remark}
  For $t=\pm 1$ we have $\loo[+]{1}(t)=(\loo[+]{1}\cdot(\colc_1\colc_2))(-t)$, so $\loo{1}$ is closed.
\end{remark}

\begin{lemma}\label{lem gam1}
  The chains $\spec{1}$ and $\act{1}(\loo{1},\pro{1})$ are homologous, i.e.,
  $$
    \exists \gam{1}\in C_2(\SCv_2(2,\pm))
    \text{ such that }
    d\gam{1}=\spec{1}-\act{1}(\loo{1},\pro{1})
  $$
\end{lemma}

\begin{proof}
  We may define, by abuse of notation, the function
  \begin{gather*}
    \spec[+]{1}:[-1,1]\rightarrow \SCv_2(2,\pm)\\
    \spec[+]{1}(t) \coloneqq \begin{cases}
      \act{1}\circ_{\colo}\push[\sgn(t)]{1}(\sgn(t)6t-1),    & \lvert t\rvert \leq 1/3  ;     \\
      \push[-\sgn(t)]{1}(\sgn(t)6t-3)\circ_{\sgn(t)}\act{1}, & 1/3\leq\lvert t\rvert\leq 2/3; \\
      \pro{1}\begin{cases}
               \circ_{-\sgn(t)}\act{1} \\
               \circ_{\sgn(t)}\act{1}
             \end{cases},                             & 2/3\leq \lvert t\rvert.
    \end{cases}
  \end{gather*}
  Due to the equations in Remark~\ref{rmk beta 01} the image of the points where this formula is ambiguous, i.e. the integer multiples of $\frac{1}{3}$, are uniquely determined when we extend by continuity on the unambiguous points. This function, as a chain, is homologous to our previous definition.

  Let $\gam[+]{1}\in C_2(\SCv_2(2,\pm))$ be defined as
  \begin{align*}
    \gam[+]{1}(t)_{\colc_i}
     & \coloneqq\min(1,1-t_2)\spec[+]{n}(t_1)_{\colc_i}+\max(0,t_2)\act{1}(\loo[+]{1}(t_1),\pro{1}), \\
    \gam[+]{1}(t)_\star
     & \coloneqq\max(-t_2,0)\spec[+]{1}(t_1)_\star+\left[0,\frac{t_2+1}{4}\right]\times\left[\frac{\star-1}{2},\frac{\star+1}{2}\right],
  \end{align*}
  for $\star\in \pm$ and $i\in \underline 2$. The following conditions guarantee that for all $t\in [-1,1]^2$ the cubes in the configurations $\gam[+]{1}(t)$ have pairwise disjoint interiors, and so are indeed elements of $\SCv(2,\pm)$:
  \begin{enumerate}
    \item In the first half of the homotopy the closed cubes stay constant and all open cubes are deformed in the first coordinate so that halfway through the homotopy they are separated from the closed cubes by the hyperplane $\{\frac{1}{4}\}\times [-1,1]$;

    \item If $t\in\left[-\frac{1}{2},\frac{1}{2}\right]$ then $\spec[+]{1}(t)_{\colc_1}$ and $\spec[+]{1}(t)_{\colc_2}$ are separated by $\left\{\frac{1}{2}\right\}\times[-1,1]$;

    \item If $t\in\left[-\frac{1}{2},\frac{1}{2}\right]$ then $\act{1}(\loo[+]{1},\pro{1})(t)_{\colc_1}$ and $\act{1}(\loo[+]{1},\pro{1})(t)_{\colc_2}$ are separated by $\left\{\frac{3}{4}\right\}\times[-1,1]$;

    \item If $\lvert t\rvert\geq \frac{1}{2}$ then $\spec[+]{1}(t)_{\colc_1}$ and $\act{1}(\loo[+]{1},\pro{1})(t)_{\colc_1}$ are in the same side of the line $[0,1]\times\{0\}$, and $\spec[+]{1}(t)_{\colc_2}$ and $\act{1}(\loo[+]{1},\pro{1})(t)_{\colc_2}$ are in the other.
  \end{enumerate}
  Note that
  $$
    d^-_{2}\gam[+]{1}=\spec[+]{1}
    ,\qquad
    d^+_{2}\gam[+]{1}=\act{1}(\loo[+]{1},\pro{1}).
  $$
  By construction for all $t\in\{-1,1\}\times[-1,1]$ we have $\gam[+]{1}(t)=\gam[+]{1}(-t_1,t_2)\cdot (\colc_1\colc_2)$, so for $\gam{1} \coloneqq \gam[+]{1}+\gam[+]{1}\cdot(\colc_1\colc_2)$ we have
  \begin{equation*}
    d\gam{1}
    =
    -d^-_{1}\gam{1}+d^+_{1}\gam{1}+d^-_{2}\gam{1}-d^+_{2}\gam{1}
    =\spec{1}-\act{1}(\loo{1},\pro{1})
    \qedhere
  \end{equation*}
\end{proof}

\subsection{Proof of the non-formality of \texorpdfstring{$\SCv_2$}{SCvor2}}
\label{sec proof non formal 2}

The chain $\spec{1}$ defined above gives rise to something similar to a nonzero type-II Massey product $\langle \pro{2}; \act{2}, \act{2} \rangle_{\mathrm{II}} \cdot (1 + (\colc_1 \colc_2))$ with the notation of~\cite{Livernet2015}.
Note however that, due to the symmetric group actions appearing outside of the Massey product, it seems that this does not define a Massey product in the proper sense.
Nevertheless, the proofs of \textcite{Livernet2015} still work with slight adaptations and we obtain:

\begin{lemma}[{\cite[Appendix A]{Vieira2018a}}]\label{Lemma NonForm Dim 2}
  If
  \begin{align*}
    \pro*{1}
             & \in C_0(\SCv_2(0, {\pm})),
             & \act*{1}
             & \in C_0(\SCv_2(1,1)),
             & \loo*{1}
             & \in C_1(\mathcal C^2(2)), \\
    \pro-{1} & \in C_1(\SCv_2(0, {\pm}))
             & \act-{1}                   & \in C_1(\SCv_2(1,1))
    ,        & \loo-{1}                   & \in C_2(\mathcal C^2(2))
  \end{align*}
  are such that
  \begin{align*}
    \pro*{1} & =\pro{1}+d \pro-{1}
    ,        & \act*{1}            & =\act{1} + d \act-{1}
    ,        & \loo*{1}            & =\loo{1} + d \loo-{1},
  \end{align*}
  then there are chains
  \begin{align*}
    \push*[-]{1}, \push*[+]{1}
     & \in C_1(\SCv_2(1,\pm)),
     & \gam*{1}
     & \in C_2(\SCv_2(2,\pm))
  \end{align*}
  such that, for $\spec*[]{1}
    \in C_1(\SCv_2(2,\pm))$
  defined from the $\push*[\star]{1}$ and $\act*{1}$ as in definition \ref{def spec1}, we have
  \begin{align*}
    d \push*[\star]{1}
     & =\pro*{1}\circ_\star\act*{1}-\act*{1}\circ_{\colo}\pro*{1} ;
    \\
    d\gam*{1}
     & =\spec*[]{1}-\act*{1}(\loo*{1},\pro*{1}).
  \end{align*}
\end{lemma}

\begin{proof}
  The chains
  \begin{align*}
    \push*[\star]{1} & = \push[\star]{1}
    +\Delta(-\act-{1}\circ_{\colo}\pro-{1}
    +\pro-{1}\circ_{\star}\act-{1})
    ; \\
    \push-[\star]{1} & =\push[\star]{1}\pr_1+\Delta(-\act-{1}\circ_{\colo}\pro-{1}
    +\pro-{1}\circ_{\star}\act-{1})\Psi_1
    \\
    \gam*{1}
                     & =\gam{1}
    +\Delta(-\push-[+]{1}\circ_-\act-{1}-\act-{1}\circ_\colo \push-[-]{1}+\act-{1}\circ_\colo \push-[+]{1}+\push-[-]{1}\circ_+\act-{1})\cdot(1+(\colc_1\colc_2))  -\Delta(\act-{1}(\loo-{1},\pro-{1}))
  \end{align*}
  satisfy the desired relations.
\end{proof}

\begin{theorem}[{\cite[Appendix A]{Vieira2018a}}]\label{NonFormality n=1}
  The operad $\SCv_2$ is not formal.
\end{theorem}

\begin{proof}
  Assume by contradiction there there is a relative dg operad $\mathcal{Q}$ and a span of quasi-isomorphisms
  $$\begin{tikzcd}
      C_\bullet(\SCv_2) & \mathcal{Q} \arrow[l, "\sim", "\phi" swap] \arrow[r, "\sim" swap, "\psi"] & H_\bullet(\SCv_2)
    \end{tikzcd}
    .$$
  Since $\phi$ is a quasi-isomorphism there are cycles
  $$
    \pro{Q} \in \mathcal{Q}(0, {\pm})_0,
    \qquad \act{Q} \in \mathcal{Q}(1,1)_0,
    \qquad \loo{Q} \in \mathcal{Q}(2)_1
  $$
  and
  $$
    \pro-{1} \in C_1(\SCv_2(0, {\pm})),
    \qquad \act-{1} \in C_1(\SCv_2(1,1)),
    \qquad \loo-{1} \in C_2(\SCv_2(2))
  $$
  such that
  $$
    \phi \pro{Q} = \pro{1}+d\pro-{1}
    ,\qquad \phi \act{Q} = \act{1} + d\act-{1}
    ,\qquad \phi \loo{Q} = \loo{1} + d\loo-{1}.
  $$
  Let $\push*[-]{1}$, $\push*[+]{1}$ and $\gam*{1}$ be the chains in the conclusion of lemma \ref{Lemma NonForm Dim 2}. By the relations
  $$
    d \push[\star]{1}
    =\pro{1}\circ_\star\act{1}-\act{1}\circ_{\colo}\pro{1}
  $$
  for $\star\in \pm$ we have that
  $$
    \phi_\ast([\pro{Q} \circ_\star \act{Q}]-[\act{Q}\circ_{\colo} \pro{Q}])
    =0,
  $$
  in $H_0(\SCv_2(1,\pm))$, and there are
  $$
    \push[-]{Q},
    \push[+]{Q}
    \in \mathcal{Q}(1,\pm)_1
  $$
  such that
  $$
    d \push[\star]{Q}
    =\pro{Q} \circ_\star\act{Q} -\act{Q} \circ_{\colo} \pro{Q}.
  $$

  Applying $\phi$ and lemma \ref{Lemma NonForm Dim 2} we obtain
  \begin{align*}
    d\phi \push[\star]{Q}
    =\phi d \push[\star]{Q}
     & =\phi\pro{Q}\circ_\star\phi\act{Q}-\phi\act{Q}\circ_{\colo}\phi\pro{Q}
    =d\push*[\star]{1}.
  \end{align*}

  Define $\spec{Q} \in \opQ(2,\pm)_1$ as in Definition \ref{def spec1}, so that $d\spec{Q}=0$.
  Since the group $H_1(\SCv_2(2,\pm))$ is generated by $[\act{1}(\loo{1},\pro{1})] = [\phi\act{Q}(\phi \loo{Q},\phi\pro{Q})]$, we have that that $H_1(\mathcal{Q}(2,\pm))$ is generated by $[\act{Q}( \loo{Q},\pro{Q})]$ and that there are
  $$
    \gam{Q}\in\mathcal{Q}(2,\pm)_2,
    \qquad \lambda\in \K
  $$
  such that
  \begin{equation}\label{Eq NF dim 2}
    \spec{Q}=\lambda \act{Q}( \loo{Q},\pro{Q})+d\gam{Q}.
  \end{equation}

  Since $H_1(\SCv_2(1,\pm))=0$ there are $\push-[-]{1},\push-[+]{1}\in C_2(\SCv_2(1,\pm))$ such that
  $$
    d \push-[-]{1}=\phi\push[-]{Q} - \push*[-]{1}
    ,\qquad
    d \push-[+]{1}=\phi\push[+]{Q} - \push*[+]{1}.
  $$
  Defining
  $$
    \spec-{1} \coloneqq \Delta(-\push-[+]{1}\circ_-\act-{1} - \act-{1}\circ_\colo \push-[-]{1} + \act-{}\circ_\colo \push-[+]{1} + \push-[-]{1}\circ_+\act-{1})\cdot(1+(\colc_1\colc_2))
  $$
  we have
  $$
    d \spec-{1}=\phi\spec{Q}-\spec*[]{1}.
  $$

  Applying $\phi$ to equation \ref{Eq NF dim 2} we get
  $$
    \spec*[]{1}+d \spec-{1}
    =
    \lambda \phi\act{Q}(\phi \loo{Q},\phi\pro{Q})+d\phi\gam{Q},
  $$
  which by lemma \ref{Lemma NonForm Dim 2} gives us
  \begin{align*}
    \spec*[]{1}
     & =\lambda \phi\act{Q}(\phi \loo{Q},\phi\pro{Q}) + d(\phi \gam{Q}  - \spec-{1}) \\
     & =\phi\act{Q}(\phi \loo{Q},\phi\pro{Q})+d\gam*{1}.
  \end{align*}
  Since $[\phi\act{Q}(\phi \loo{Q},\phi\pro{Q})]$ generates $H_1(\SCv_2(2,\pm))$ we conclude that $\lambda=1$.

  For degree reasons $\psi\push[-]{Q}=\psi\push[+]{Q}=0$, thus $\psi\spec[+]{Q}=0$. Similarly $\psi\gam{Q}=0$. Since $\psi$ is a quasi-isomorphism there are
  $$
    \lambda_\mu, \lambda_\alpha,\lambda_\ell\in \K^\times
  $$
  such that
  \begin{align*}
    \psi \pro{Q} & = \lambda_\mu[\pro{1}],
                 & \psi \act{Q}            & = \lambda_\alpha[\act{1}],
                 & \psi \loo{Q}            & = \lambda_\ell [\loo{1}].
  \end{align*}

  Applying $\psi$ to \ref{Eq NF dim 2} we get
  $$
    \lambda \lambda_\alpha \lambda_\ell \lambda_\mu [\act{1}(\loo{1},\pro{1})]=0
  $$
  which implies $\lambda=0$, a contradiction.
\end{proof}

\section{Non-formality in dimension 3}
\label{sec non-formal 3}
Before moving on to the general case, let us first deal with the special case $n = 2$ of the proof.
This case allows us to draw actual pictures that illustrate the general procedure, but there is no difference in the actual proof with Section~\ref{sec non-formal all}.

\subsection{Definition of the basic squares}
\label{sec:basic-squares}

In order to build the nontrivial 2-chain that proves the non-formality of $\SCv_3$, we will need to define several kinds of basic squares.

The first kind of basic squares is similar to the basic paths defined in Section~\ref{sec:def-basic-paths}.
Just like the path $\push[+]{1}$ (resp.\ $\push[-]{1}$) were defined to ``push'' the closed input towards the open input labeled by $+$ (resp.\ $-$), we will define four operations $\push[\star_1 \star_2]{2}$, for $\star_1, \star_2 \in \pm$, that ``push'' the closed input towards corners of the unit square.

Recall from Convention~\ref{conv pm} that we denote $\pm^2 = \{++,+-,-+,--\}$.
Recall moreover the inclusions $\inc[k]{2}$ from Definition~\ref{def inc}.

\begin{definition}\label{def pro2}
  The four-fold product $\pro{2} \in \SCv(0, \pm^2)$ is defined as:
  \begin{equation}
    \pro{2} \coloneqq \inc[\{1\}]{2} \pro{1} \bigl( \inc[\{2\}]{2}\pro{1}, \inc[\{2\}]{2} \pro{1} \bigr)
    =
    \raisebox{-.5\height}{\includegraphics[width=.4\textwidth]{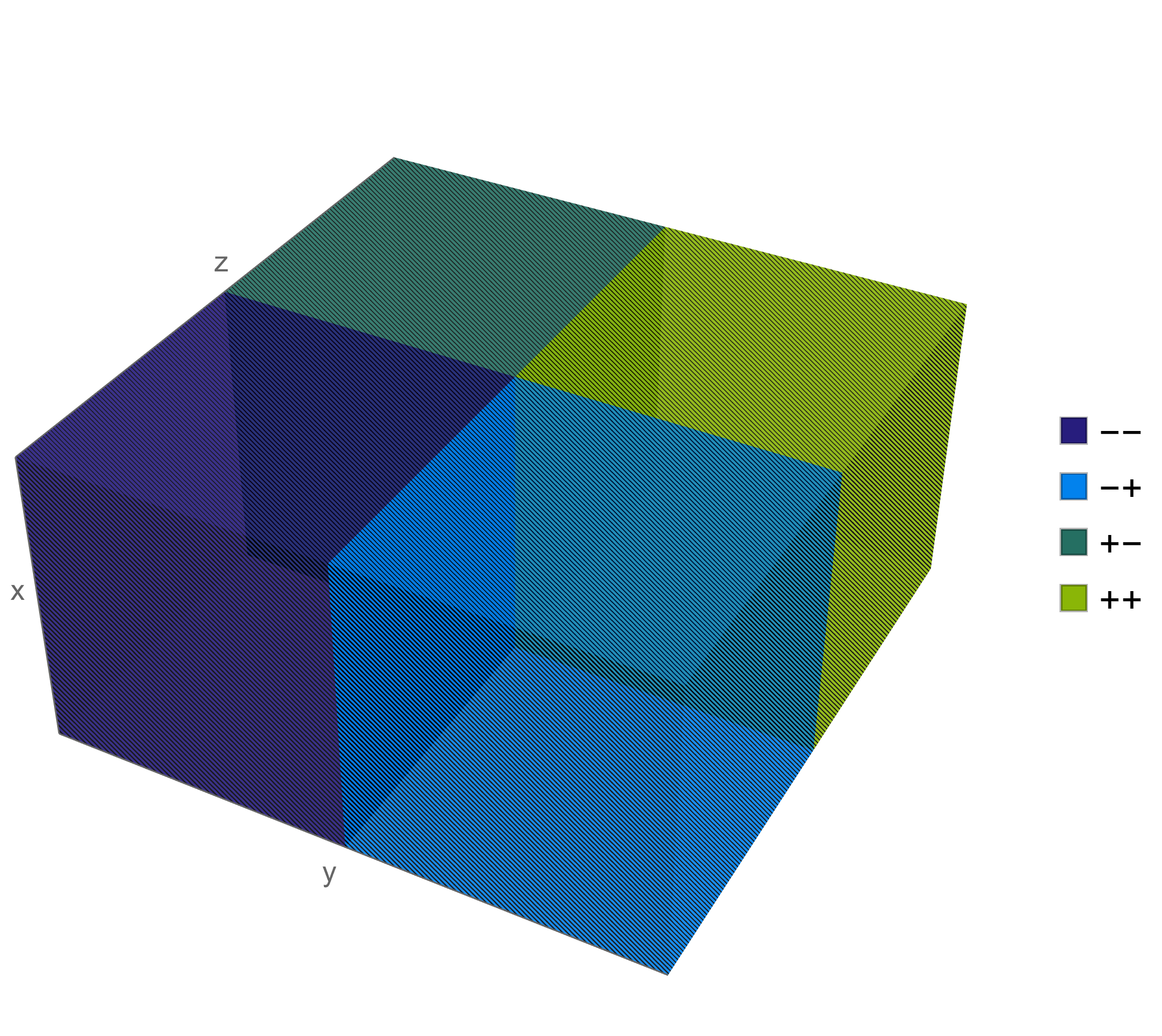}}
    .
  \end{equation}
\end{definition}

\begin{remark}
  We use Convention~\ref{conv compos pm} to find the labels of the inputs of $\pro{2}$.
  Given an index $\star = \star_1 \star_2 \in \pm^2$, the corresponding component of $\pro{2}$ is given by:
  \begin{equation}
    (\pro{2})_{\star_1 \star_2} =
    [0, 1]
    \times \Bigl[\frac{\star_1-1}{2}, \frac{\star_1+1}{2} \Bigr]
    \times \Bigl[\frac{\star_2-1}{2}, \frac{\star_2+1}{2} \Bigr].
  \end{equation}
\end{remark}

\begin{definition}\label{def act2}
  The action $\act{2} \in \SCv_3(1, \pm^0)$ is defined by:
  \begin{equation}
    \act{2} =
    \begin{cases}
      [1/2, 1] \times [-1, 1]^2
       & (\colc)  \\
      [0, 1/2] \times [-1, 1]^2
       & (\colo).
    \end{cases}
  \end{equation}
\end{definition}

We can now define the higher-dimensional generalizations of the paths $\push[\pm]{1}$ from Section~\ref{sec:def-basic-paths}.
For notational convenience, we introduce the following helpers (see Figure~\ref{fig:plot-sigma-tau}).

\begin{definition}
  For $\star \in \pm$ and $t \in [-1, 1]$, let:
  \begin{equation}
    \sigma_{\star}(t)
    \coloneqq
    \begin{cases}
      \min(0, t), & \text{if } \star = +; \\
      -1,         & \text{if } \star = -.
    \end{cases}
  \end{equation}
  Moreover, given $k \in \mathbb{N}$, $\star = (\star_1, \dots, \star_k) \in \pm^k$ and $t = (t_1, \dots, t_k) \in [-1, 1]^k$, let:
  \begin{equation}
    \tau_{\star_1, \dots, \star_k} \bigl(t_1, \dots, t_k\bigr)
    \coloneqq
    \max \Bigl( \{1/2\} \cup \Bigl\{ \frac{1+t_i}{2} \mid 1 \leq i \leq k, \, \star_i = - \Bigr\} \Bigr).
  \end{equation}
\end{definition}

\begin{figure}[htbp]
  \centering
  \subcaptionbox{Plots of $\sigma_\pm$ and $\tau_\pm$.}[.49\linewidth][c]{\includegraphics[width=\linewidth]{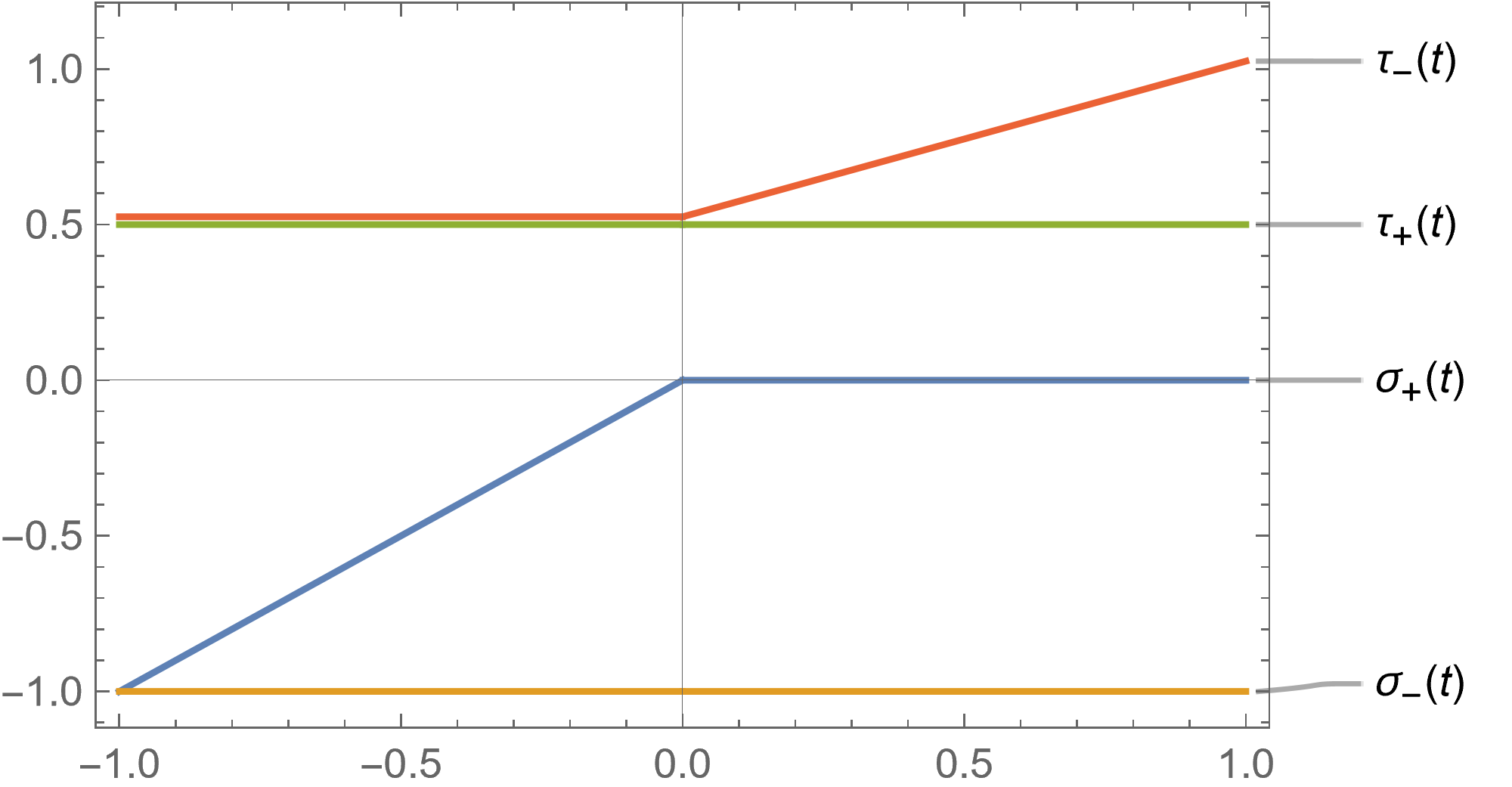}}
  \subcaptionbox{Plot of $\tau_{- -}$}[.49\linewidth][c]{\includegraphics[width=\linewidth]{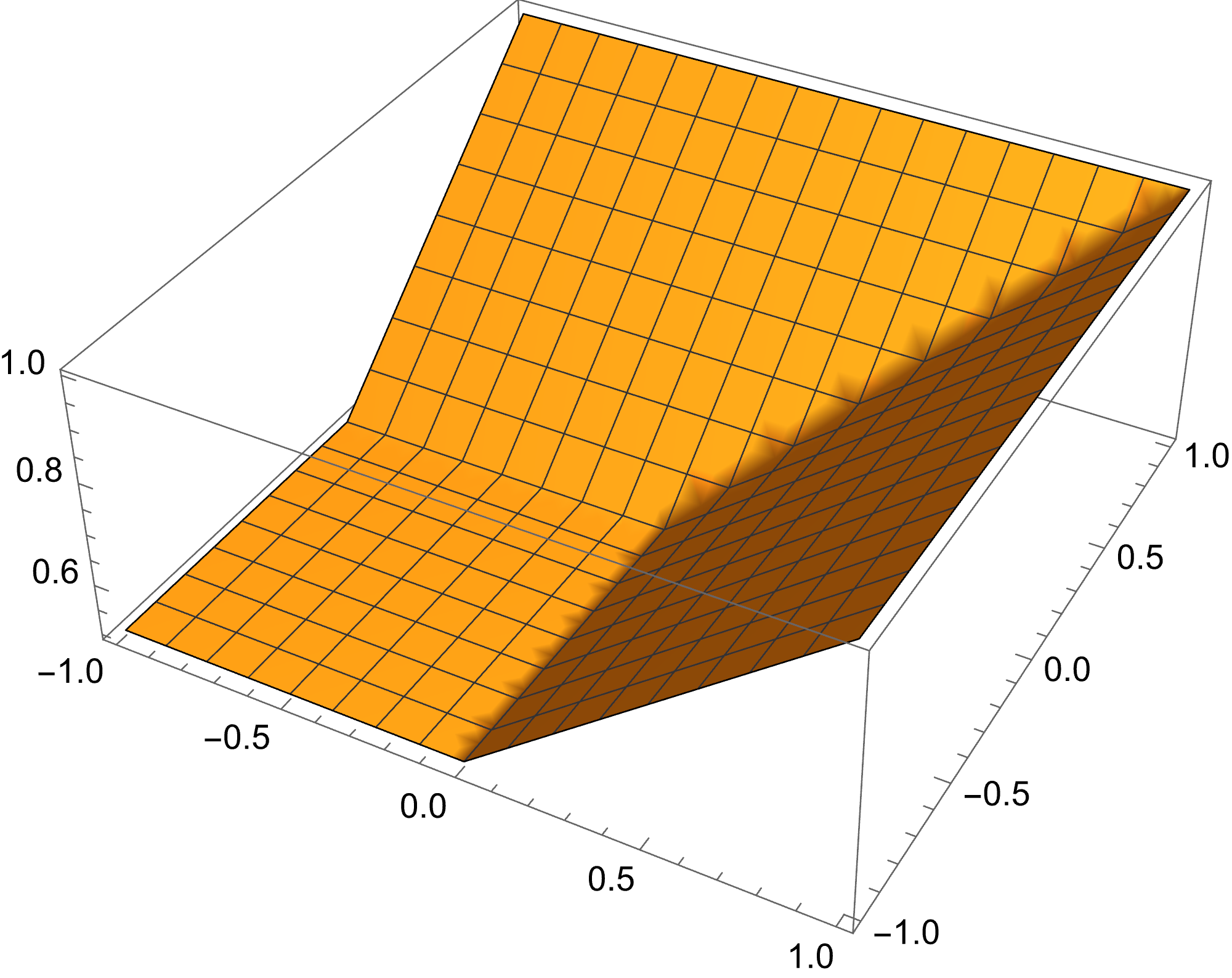}}
  \caption{Plots of our helper functions.}
  \label{fig:plot-sigma-tau}
\end{figure}

\begin{remark}
  \label{rmk shortcut}
  Using this notational shortcut, we have, for $\star \in \pm$ and $0 \leq t \leq 1$, that the components of $\push[\star]{1}(t) = \{ \push[\star]{1}(t)_-, \push[\star]{1}(t)_+, \push[\star]{1}(t)_{\colc} \} \in \SCv_2(1, \pm)$ are respectively given by:
  \begin{align*}
    \bigl(\push[\star]{1}(t)\bigr)_{\colc}
     & = [1/2, 1] \times [\sigma_\star(t), -\sigma_{-\star}(t)],
    \\
    \bigl(\push[\star]{1}(t)\bigr)_+
     & = [0, \tau_\star(t)] \times [0, 1],
     &
    \bigl(\push[\star]{1}(t)\bigr)_-
     & = [0, \tau_{-\star}(t)] \times [-1, 0].
  \end{align*}
  The two equations for the open inputs can be condensed further into, where $\star' \in \pm$:
  \begin{equation*}
    \bigl(\push[\star]{1}(t)\bigr)_{\star'} = [0, \tau_{\star \star'}(t)] \times \Bigl[ \frac{\star'-1}{2}, \frac{\star'+1}{2} \Bigr].
  \end{equation*}
\end{remark}

\begin{definition}\label{def push2}
  Let $(\star_1, \star_2) \in \pm^2$.
  The chain $\push[\star_1 \star_2]{2} : [-1, 1]^2 \to \SCv_3(1, \pm^2)$ is defined, for $-1 \leq t_1, t_2 \leq 1$, as follows (see Figure~\ref{fig:beta-plus-plus}):
  \begin{itemize}
    \item The closed component is given by the cube:
          \begin{equation}
            \bigl(\push[\star_1 \star_2]{2}(t_1, t_2)\bigr)_{\colc}
            \coloneqq
            [1/2, 1]
            \times [\sigma_{\star_1}(t_1),  - \sigma_{-\star_1}(t_1)]
            \times [\sigma_{\star_2}(t_2),  - \sigma_{-\star_2}(t_2)].
          \end{equation}
    \item The open component indexed by a pair $(\star'_1, \star'_2) \in  \pm^2$ is given by the cube:
          \begin{equation}
            \bigl(\push[\star_1 \star_2]{2}(t_1, t_2)\bigr)_{\star'_1 \star'_2}
            \coloneqq
            [0, \tau_{\star_1 \star'_1, \star_2 \star'_2}(t_1, t_2)],
            \times \Bigl[ \frac{\star'_1-1}{2}, \frac{\star'_1+1}{2} \Bigr]
            \times \Bigl[ \frac{\star'_2-1}{2}, \frac{\star'_2+1}{2} \Bigr].
          \end{equation}
  \end{itemize}
\end{definition}

\begin{figure}[htbp]
  \centering
  \includegraphics[width=.7\textwidth]{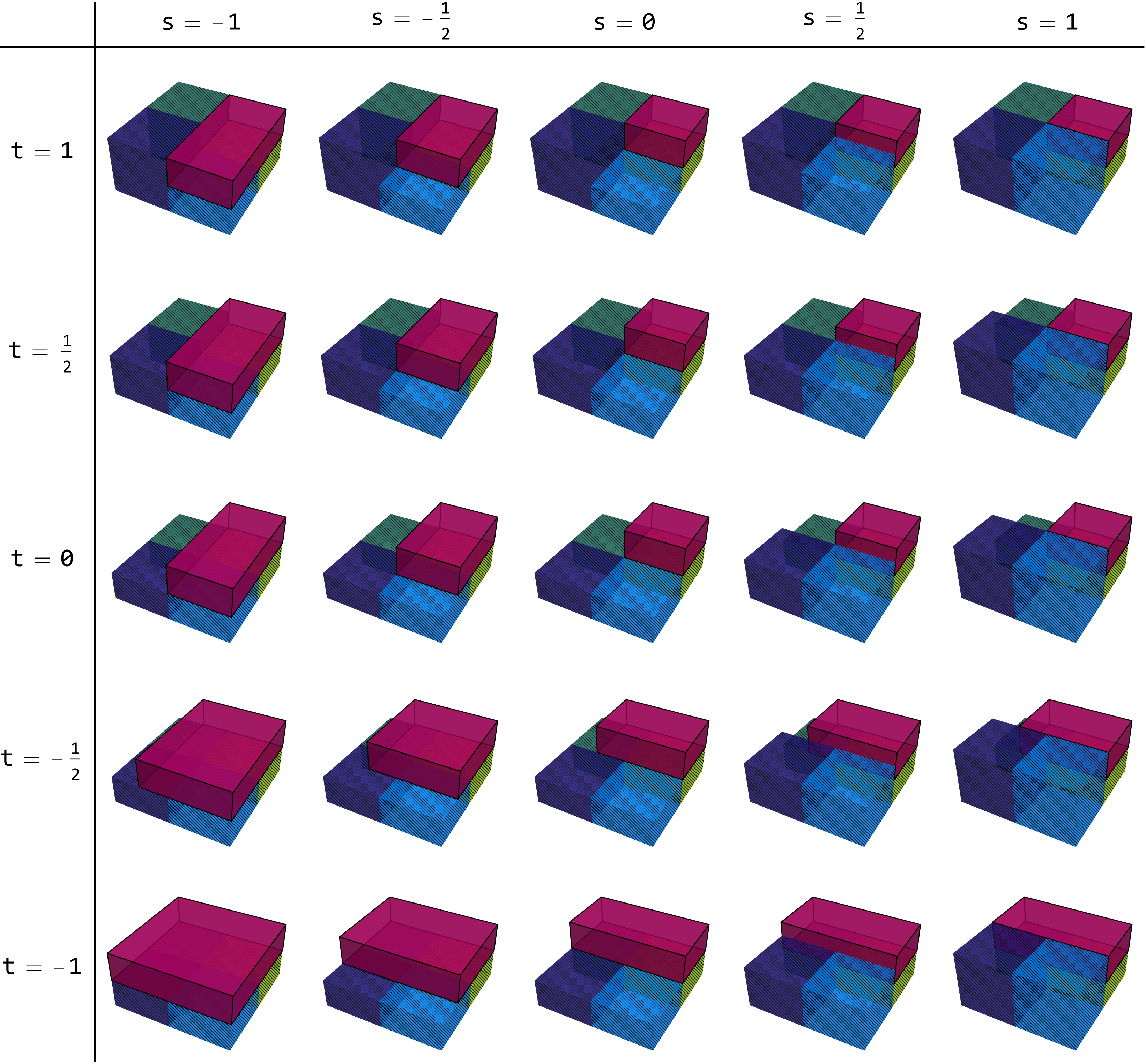}
  \caption{The 2-chain $\push[++]{2} \in C_2(\SCv_3(1, \pm^2))$.}
  \label{fig:beta-plus-plus}
\end{figure}

This 2-chain is built such that $\push[\star_1 \star_2]{2}(1, 1) = \pro{2} \circ_{\star_1 \star_2} \act{2}$ and $\push[\star_1 \star_2]{2}(0, 0) = \act{2} \circ_{\colo} \pro{2}$.
More precisely, we have the following lemma.

\begin{lemma}
  Given $\star = (\star_1, \star_2) \in \pm^2$, the boundary of $\push[\star_1 \star_2]{2}$ decomposes as:
  \begin{equation*}
    \begin{tikzcd}[scale cd=0.95,
        column sep = tiny,
        row sep = tiny,
        ampersand replacement = \&,
        labels={font=\tiny}
      ]
      \begin{aligned}
        \inc[\{2\}]{2}\pro{1}\!
        \begin{cases}
          \circ_{-\star_2} \inc[\{1\}]{2} \pro{1} \\
          \circ_{\star_2} \act{2}(\id_{\colc},\inc[\{1\}]{2} \pro{1})
        \end{cases}
      \end{aligned}
      \arrow[rr, "{
      \begin{aligned}
        \inc[\{2\}]{2} \pro{2}\!
        \begin{cases}
          \circ_{-\star_2} \inc[\{1\}]{2} \pro{1} \\
          \circ_{\star_2} \inc[\{1\}]{2} \push[\star_1]{1}(t_1)
        \end{cases}
      \end{aligned}
      }"]
      \& \&
      \pro{2} \circ_{\star_1 \star_2} \act{2}
      \\
      {} \&
      \push[\star_1 \star_2]{2}(t_1, t_2)
      \&
      \\
      \act{2} \circ_{\colo} \pro{2}
      \arrow[rr, "{
      \begin{aligned}
        \inc[\{1\}]{2} \push[\star_1]{1}(t_1)\!
        \begin{cases}
          \circ_{-\star_1} \inc[\{2\}]{2} \pro{1} \\
          \circ_{\star_1} \inc[\{2\}]{2} \pro{1}
        \end{cases}
      \end{aligned}
      }" swap]
      \arrow[uu, "{
      \begin{aligned}
        \inc[\{2\}]{2} \push[\star_2]{1}(t_2)\!
        \begin{cases}
          \circ_{-\star_2} \inc[\{1\}]{2}\pro{1} \\
          \circ_{\star_2} \inc[\{1\}]{2}\pro{1}
        \end{cases}\!\!\!\!\!\!\!
      \end{aligned}
      }"]
      \& \&
      \begin{aligned}
        \inc[\{1\}]{2}\pro{1}\!
        \begin{cases}
          \circ_{-\star_1} \inc[\{2\}]{2} \pro{1} \\
          \circ_{\star_1} \act{2}(\id_{\colc},\inc[\{2\}]{2} \pro{1})
        \end{cases}
      \end{aligned}
      \arrow[uu, "{
      \begin{aligned}
        \inc[\{1\}]{2} \pro{1}\!
        \begin{cases}
          \circ_{-\star_1} \inc[\{2\}]{2} \pro{1} \\
          \circ_{\star_1} \inc[\{2\}]{2} \push[\star_2]{1}(t_2)
        \end{cases}
      \end{aligned}
      }" swap]
    \end{tikzcd}
  \end{equation*}
\end{lemma}

\subsection{Sketch of proof of the non-formality of \texorpdfstring{$\SCv_3$}{SCvor3}}
\label{sec sketch non formal 3}

Using the above data, we can now define the following cubical $2$-chain $\spec[+]{2} \in \SCv_3(2, \pm^2)$ by gluing and composing together several basic chains:

\NewDocumentCommand{\ourstack}{mm}{
  \!\!\left\lbrace \begin{matrix*}[l] \! #1 \\ \! #2 \end{matrix*} \right.
}

\begin{equation}\label{eq huge square}
  \begin{tikzpicture}[shorten >=3pt, shorten <=3pt]
    \matrix [
      matrix of math nodes,
      row sep = 0pt,
      column sep = 0pt,
      nodes={font=\tiny, inner sep = 0pt, minimum width=7em, minimum height=5ex},
    ] (M)
    {
      \pro{2}\ourstack{\circ_{+-}\act{2}}{\circ_{-+}\act{2}}
       &
      \pro{\{2\}}\ourstack{\circ_-\push[+]{\{1\}}\!\!\pr_1}{\circ_+\push[-]{\{1\}}\!\Phi_{\emptyset,\{2\}}}
       &
      \pro{\{2\}}\ourstack{\circ_-\act{2}\circ_\colo\pro{\{1\}}}{\circ_+\push[-]{\{1\}}\!\Phi_{\{1\},\{2\}}}
       &
      \pro{\{2\}}\ourstack{\circ_-\act{2}\circ_\colo\pro{\{1\}}}{\circ_+\push[+]{\{1\}}\!\Phi_{\{1\},\{2\}}}
       &
      \pro{\{2\}}\ourstack{\circ_-\push[-]{\{1\}} \!\!\pr_1}{\circ_+\push[+]{\{1\}}\!\Phi_{\emptyset,\{2\}}}
       &
      \pro{2}\ourstack{\circ_{--}\act{2}}{\circ_{++}\act{2}}
      \\
      \pro{\{1\}}\ourstack{\circ_+\push[-]{\{2\}}\!\!\pr_2}{\circ_-\push[+]{\{2\}}\!\Phi_{\emptyset,\{1\}}}
       &
      \push[+-]{\underline 2}\circ_{-+}\act{2}
       &
      \push[-]{\{2\}}\!\!\pr_2\ourstack{\circ_- \pro{\{1\}}}{\circ_+\push[-]{\{1\}}\!\!\pr_1}
       &
      \push[-]{\{2\}}\!\!\pr_2
      \ourstack{\circ_-\pro{\{1\}}}{\circ_+\push[+]{\{1\}}\!\!\pr_1}
       &
      \push[--]{\underline 2}\circ_{++}\act{2}
       & \pro{\{1\}}\ourstack{\circ_-\push[-]{\{2\}} \!\!\pr_2}{\circ_+\push[+]{\{2\}}\!\Phi_{\emptyset,\{1\}}}
      \\
      \pro{\{1\}}\ourstack{\circ_+\act{2}\circ_\colo\pro{\{2\}}}{\circ_-\push[+]{\{2\}}\!\Phi_{\{2\},\{1\}}}
       &
      \push[+]{\{1\}}\!\!\pr_1\ourstack{\circ_+\pro{\{2\}}}{\circ_- \push[+]{\{2\}}\!\!\pr_2}
       &
      \act{2}\circ_{\colo}\push[-+]{\underline 2}
       &
      \act{2}\circ_{\colo}\push[++]{\underline 2}
       &
      \push[-]{\{1\}}\!\!\pr_1
      \ourstack{\circ_-\pro{\{2\}}}{\circ_+\push[+]{\{2\}}\!\!\pr_2}
       &
      \pro{\{1\}}\ourstack{\circ_-\act{2}\circ_\colo\pro{\{2\}}}{\circ_+\push[+]{\{2\}}\!\Phi_{\{2\},\{1\}}}
      \\
      \pro{\{1\}}\ourstack{\circ_+\act{2}\circ_\colo\pro{\{2\}}}{\circ_-\push[-]{\{2\}}\!\Phi_{\{2\},\{1\}}}
       &
      \push[+]{\{1\}}\!\!\pr_1\ourstack{\circ_+\pro{\{2\}}}{\circ_-\push[-]{\{2\}}\!\!\pr_2}
       &
      \act{2}\circ_{\colo}\push[--]{\underline 2}
       &
      \act{2}\circ_{\colo}\push[+-]{\underline 2}
       &
      \push[-]{\{1\}}\!\!\pr_1\ourstack{\circ_- \pro{\{2\}}}{\circ_+\push[-]{\{2\}}\!\!\pr_2}
       &
      \pro{\{1\}}\ourstack{\circ_-\act{2}\circ_\colo\pro{\{2\}}}{\circ_+\push[-]{\{2\}}\!\Phi_{\{2\},\{1\}}}
      \\
      \pro{\{1\}}\ourstack{\circ_+\push[+]{\{2\}}\!\!\pr_2}{\circ_-\push[-]{\{2\}}\!\Phi_{\emptyset,\{1\}}}
       &
      \push[++]{\underline 2}\circ_{--}\act{2}
       &
      \push[+]{\{2\}}\!\!\pr_2\ourstack{\circ_+\pro{\{1\}}}{\circ_-\push[-]{\{1\}}\!\!\pr_1}
       &
      \push[+]{\{2\}}\!\!\pr_2\ourstack{\circ_+\pro{\{1\}}}{\circ_-\push[+]{\{1\}}\!\!\pr_1}
       &
      \push[-+]{\underline 2}\circ_{+-}\act{2}
       &
      \pro{\{1\}}\ourstack{\circ_-\push[+]{\{2\}}\!\!\pr_2}{\circ_+\push[-]{\{2\}}\!\Phi_{\emptyset,\{1\}}}
      \\
      \pro{2}\ourstack{\circ_{++}\act{2}}{\circ_{--}\act{2}}
       &
      \pro{\{2\}}\ourstack{\circ_+\push[+]{\{1\}} \!\!\pr_1}{\circ_-\push[-]{\{1\}}\!\Phi_{\emptyset,\{2\}}}
       &
      \pro{\{2\}}\ourstack{\circ_+\act{2}\circ_\colo\pro{\{1\}}}{\circ_-\push[-]{\{1\}}\!\Phi_{\{1\},\{2\}}}
       &
      \pro{\{2\}}\ourstack{\circ_+\act{2}\circ_\colo\pro{\{1\}}}{\circ_-\push[+]{\{1\}}\!\Phi_{\{1\},\{2\}}}
       &
      \pro{\{2\}}\ourstack{\circ_+\push[-]{\{1\}} \!\!\pr_1}{\circ_-\push[+]{\{1\}}\!\Phi_{\emptyset,\{2\}}}
       &
      \pro{2}\ourstack{\circ_{-+}\act{2}}{\circ_{+-}\act{2}}
      \\
    };
    \foreach \i/\di in {1/->, 2/->, 3/->, 4/<-, 5/<-, 6/<-}
    \foreach \j/\dj in {1/<-, 2/<-, 3/<-, 4/->, 5/->, 6/->}
      {
        \draw[\dj] (M-\i-\j.south west) -- (M-\i-\j.south east);
        \draw[\di] (M-\i-\j.south west) -- (M-\i-\j.north west);
      }
    \foreach \i/\di in {1/<-, 2/<-, 3/<-, 4/->, 5/->, 6/->}
    \draw[\di] (M-\i-6.north east) -- (M-\i-6.south east);
    \foreach \j/\dj in {1/<-, 2/<-, 3/<-, 4/->, 5/->, 6/->}
    \draw[\dj] (M-1-\j.north west) -- (M-1-\j.north east);
  \end{tikzpicture}
\end{equation}

\begin{lemma}
  Let $\spec[+]{2}$ be the above 2-chain.
  Then $\spec[+]{2} \cdot (1-(\colc_1 \, \colc_2))$ is closed and homologous to $\act{2}(\loo{2},\pro{2})$.
\end{lemma}

\begin{theorem}
  The operad $\SCv_3$ is not formal.
\end{theorem}

The proofs of the above lemma and corollary are simply special cases of the proofs of Section~\ref{sec non-formal all}.
As explained in the introduction of this section, we only included this case to give visual pictures for the chains that we will build there.

\section{Non-formality in all dimensions}
\label{sec non-formal all}

\subsection{Definition of the basic \texorpdfstring{$n$}{n}-chains}

We are now ready to give the definitions of the basic $n$-chains (analogous to $\push[\star]{1}$ and $\spec[\star]{1}$ above) in $\SCv_{n+1}$.
First, let us define the multi-fold product inductively (recalling that $\pro{1}$ is found in Definition~\ref{def pro1}):
\begin{definition}\label{def proN}
  Given $n \geq 2$, we (inductively) define $\pro{n} \in C_0 \SCv_{n+1}(0, \pm^n)$ by:
  \begin{equation}
    \pro{n}
    \coloneqq
    \inc[\underline{n-1}]{n}\pro{1}
    \bigl( \inc[\{n\}]{n} \pro{n-1}, \; \inc[\{n\}]{n} \pro{n-1} \bigr).
  \end{equation}
  Concretely, given $\star \in \pm^n$, the corresponding component of $\pro{n}$ is given by:
  \begin{equation}
    (\pro{n})_{\star} =
      [0,1]\times
    \prod_{j=1}^{n} \bigl[ \frac{\star_j-1}{2}, \frac{\star_j+1}{2}\bigr]
    =
    [0,1]\times
    \prod_{j=1}^n
    \begin{cases}
      [-1,0], & \star_j=- \\
      [0,1],  & \star_j=+
    \end{cases}
  \end{equation}
\end{definition}

Moreover, we define the ``action'' as follows:

\begin{definition}
  \label{def actn}
  Given $n \geq 2$, let $\act{n} \in C_0 \SCv_{n+1}(1, \pm^0)$ be defined by:
  \begin{equation*}
    \act{n} =
    \begin{cases}
      [1/2, 1] \times [-1, 1]^n & (\colc); \\
      [0, 1/2] \times [-1, 1]^n & (\colo).
    \end{cases}
  \end{equation*}
\end{definition}

We are now ready to give a definition for the analogue of $\push[\star_1\star_2]{2}$ from Definition~\ref{def push2} in every dimension.

\begin{definition}\label{def pushn}
  Let $\star \in \pm^n$ be a string of $n$ pluses and minuses.
  Let the $n$-chain $\push[\star]{n} \in C_n\SCv_{n+1}(1, \pm^n)$ as follows, where $t = (t_1, \dots, t_n) \in [-1, 1]^n$:
  \begin{itemize}
    \item The closed component $(\push[\star]{n})_{\colc}$ is given by:
          \begin{equation*}
            (\push[\star]{n})_{\colc}(t)
            = \bigl[ 1/2, 1 \bigr]
            \times \prod_{i=1}^n \bigl[ \sigma_{\star_i}(t_i),  - \sigma_{-\star_i}(t_i) \bigr].
          \end{equation*}
    \item Given $\star' \in \pm^n$, the corresponding open component is given by:
          \begin{equation*}
            (\push[\star]{n})_{\star'}(t)
            = [0, \tau_{\star\star'}(t)]
            \times \prod_{i=1}^n \bigl[ \frac{\star'_i - 1}{2}, \frac{\star'_i+1}{2} \bigr].
          \end{equation*}
  \end{itemize}
\end{definition}

\begin{convention}\label{conv ab0}
  For the sake of consistency, let $\pro{0} \coloneqq \id \in C_0(\SCv_1(0, \pm^0))$ to be the ``one-fold product'' and $\push{0} = \act{0} \in C_0(\SCv_1(1, \pm^0))$ to be the action.
  They are respectively given by:
  \begin{align*}
    \pro{0}
     & = [0, 1],
     & \act{0}
     & =
    \begin{cases}
      [1/2, 1] & (\colc)  \\
      [0, 1/2] & (\colo).
    \end{cases}
  \end{align*}
\end{convention}

\begin{definition}
  Let $\loo[+]{n}\in \C_n(2)$ be defined as
  \begin{multline*}
    \loo[+]{n}(t)
    \coloneqq \Biggl\{
    \Bigl[\frac{1-\max\{\lvert 2t_j\rvert\mid j\in\underline{n} \}}{4},\frac{3-\max\{\lvert 2t_j\rvert\mid j\in\underline{n} \}}{4}\Bigr]
    \times \prod_{i=1}^n \Bigl[\frac{-1-2t_i}{4},\frac{1-2t_i}{4}\Bigr],
    \\
    \Bigl[\frac{-3+\max\{\lvert 2t_j\rvert\mid j\in\underline{n} \}}{4},\frac{-1+\max\{\lvert 2t_j\rvert\mid j\in\underline{n} \}}{4}\Bigr]
    \times \prod_{i=1}^n \Bigl[\frac{-1+2t_i}{4},\frac{1+2t_i}{4}\Bigr]
    \Biggr\}.
  \end{multline*}
  For $t\in\partial[-1,1]^n$ we have $\loo[+]{n}(t)=(\loo[+]{n}\cdot(12))(-t)$, so we get the closed element
  \begin{equation}\label{eq loon}
    \loo{n} \coloneqq \loo[+]{n}-(-1)^n\loo[+]{n}\cdot(12)\in C_n(\C_{n+1}(2))
  \end{equation}
  whose homology class generates $H_n(\C_{n+1}(2))$.
\end{definition}

\begin{notation}
  Given $n \geq 1$, $S\subseteq \underline{n}$, and $\star\in\pm^S$, we use the notation:
  \begin{align*}
    \push[\star]{S}
     & \coloneqq \inc[\underline{n}\setminus S]{n}\push[\star]{\lvert S\rvert}
    \in C_{|S|} \SCv_{n+1}(1, \pm^S),
     &
    \pro{S}
     & \coloneqq \inc[\underline{n}\setminus S]{n}\pro{\lvert S\rvert}
    \in \SCv_{n+1}(0, \pm^S).
  \end{align*}
\end{notation}

\begin{remark}
  The parameter $n$ is implicit in the above notation, since $S$ is not just a set but a subset of $\underline{n}$.
  However, it is important to note that $\push[\star]{S}$ and $\pro{S}$ depend on $n$.
\end{remark}

\begin{notation}
  For $\star\in \pm^n$, let $-\star\in \pm^n$ be such that if $i\in\underline{n}$ then $(-\star)_i=-(\star_i)$.
  Moreover, for $S \subseteq \underline{n}$ and $\star\in\pm^n$, let $\star_S\in\pm^{\lvert S\rvert}$ as the projection onto the coordinates in $S$.
\end{notation}

Note that if $\star \in \pm^n$, then $\star_{\underline{n}} = \star$ and $\star_{\emptyset} = \colo$ is the empty sequence (Convention~\ref{conv Qkl}).
For the next definition, recall the maps $\Phi_{S,T}$ from Definition~\ref{def Phi}.

\begin{definition}\label{def spec n}
  Let $n \geq 1$ and $S, T \subseteq \underline{n}$ be such that $S \cap T = \emptyset$.
  We then define the following cubical $n$-chains in $\SCv_{n+1}(2, \pm^n)$:
  \begin{align}\label{subEtas}
    \spec[\star,S,\emptyset]{n}
     & \coloneqq
    \Biggl(
    \push[-\star_{\underline{n}\setminus S}]{\underline{n}\setminus S}\pr_{\underline{n}\setminus S}
    \underset{\substack{\star'\in \pm^{\underline{n}\setminus S}
    \\
        \star'\neq \star_{\underline{n}\setminus S}}}{\bigcirc}\pro{S}
    \Biggr)
    \circ_{\star_{\underline{n}\setminus S}}\push[\star_{S}]{S}\pr_S
    ;
    \\
    \spec[\star,S,T]{n}
     & \underset{\mathclap{T \neq \emptyset}}{\coloneqq}
    \Bigl(
    \pro{T}\underset{\substack{\star'\in\pm^T \\ \star'\neq \star_T,-\star_T}}{\bigcirc} \pro{\underline{n}\setminus T}
    \Bigr)
    \begin{cases}
      \circ_{-\star_T}
      \left(\push[-\star_{\underline{n}\setminus (S\sqcup T)}]{\underline{n} \setminus (S\sqcup T)}\pr_{\underline{n} \setminus (S\sqcup T)}
      \underset{\star''\in\pm^{\underline{n}\setminus (S\sqcup T)}}{\bigcirc}\pro{S}\right)
      \\
      \circ_{\star_T}
      \push[\star_{\underline{n}\setminus T}]{\underline{n}\setminus T}\Phi_{S,T}
    \end{cases}
    ;
    \\
    \label{EtaPlus}
    \spec[+]{n}
     & \coloneqq \sum_{\star\in \pm^n}\sgn\star
    \sum_{\substack{S,T\subset \underline{n} \\ S\cap T=\emptyset}}
    \spec[\star,S,T]{n};
    \\
    \label{eq specn}
    \spec{n}
     & \coloneqq \spec[+]{n}-(-1)^n\spec[+]{n}\cdot(\colc_1\colc_2).
  \end{align}
\end{definition}

\begin{remark}
  Recall that $\push{0} = \act{0}$ (Convention~\ref{conv ab0}), so that $\push{\emptyset} = \inc[\underline{n}]{n} \act{0} = \act{n}$.
  It follows that when $S = T = \emptyset$, then $\spec[\star,\emptyset,\emptyset]{n} = \push[-\star]{n} \circ_{\star} \push{\emptyset}$, whereas if $S = \underline{n}$, then $\spec[\star,\underline{n},\emptyset]{n} = \push{\emptyset} \circ_{\colo} \push[\star]{n}$.
  For $n=2$, these are precisely the chains that appear in the center of the diagram~\eqref{eq huge square}, as well the corners of the inner $4 \times 4$ square of that diagram.
\end{remark}

\begin{example}
  Let $n = 1$, so that the only possibilities for the couple $(S,T)$ are $(\emptyset,\emptyset)$, $(\{1\}, \emptyset)$ and $(\emptyset, \{1\})$.
  As remarked above, we have, for $\star \in \pm$:
  \begin{align*}
    \spec[\star, \emptyset, \emptyset]{1}
     & = \push[-\star]{1} \circ_{\star} \act{1},
     & \spec[\star, \underline{1}, \emptyset]{1}
     & = \act{1} \circ_{\colo} \push[\star]{1}.
  \end{align*}
  Moreover, we have that, for $t = (t_1, t_2) \in [-1, 1]^2$:
  \begin{equation*}
    \spec[\star,\emptyset,\underline{1}]{1}(t_1, t_2)
    = \left(\pro{1} \begin{cases}
        \circ_{-\star} \act{1} \\
        \circ_{\star} \act{1}
      \end{cases} \right)(\Phi_{\emptyset, \underline{1}}(t_1, t_2))
    = \pro{1} \begin{cases}
      \circ_{-\star} \act{1} \\
      \circ_{\star} \act{1}
    \end{cases}.
  \end{equation*}
  In other words, the chain is constant (its value does not depend on $t$).
  While these constant chains are irrelevant in $\SCv_2$, we still write them down in the general formula for $\SCv_{n+1}$.
  The signed sum of the $\spec[\star,S,T]{1}$ recovers the 1-chain $\spec[+]{1}$ defined in Section~\ref{sec non formal 2}, up to these two constant chains.
  In higher dimension, the $\spec[\star,S,T]{n}$ are chains where cubes go ``back and forth'' but remain constant in some of the coordinates.
\end{example}

\subsection{Proof of the non-formality of \texorpdfstring{$\SCv_{n+1}$}{SCvor n+1}}
\label{sec proof non-formal all}

\begin{lemma}\label{Lemma Spec is closed an homologous to loop} The chain $\spec{n}$ is closed and homologous to $\act{n}(\loo{n},\pro{n})$.
\end{lemma}

\begin{proof}
  Let us recall that the boundary of $\push[\star]{n}$ is given by the signed sum of its facets:
  \begin{equation*}
    d \push[\star]{n}
    =\sum_{i\in \underline{n}}
    (-1)^{i}(d^-_i\push[\star]{n}-d^+_i\push[\star]{n}),
  \end{equation*}
  where the facets are given by (for $i \in \underline{n}$):
  \begin{align*}
    d^-_i\push[\star]{n}
     & =
    \push[\star_{\underline{n}\setminus\{i\}}]{\underline{n}\setminus\{i\}}\pr_{\underline n\setminus \{i\}}
    \underset{\star'\in\pm^{\underline{n}\setminus\{i\}}}{\bigcirc}\pro{\{i\}},
     &
    d^+_i\push[\star]{n}
     & =\pro{\{i\}}\begin{cases}
                     \circ_{-\star_i}\pro{\underline{n}\setminus\{i\}}
                     \\
                     \circ_{\star_i} \push[\star_{\underline{n}\setminus\{i\}}]{\underline{n}\setminus\{i\}}\pr_{\underline n\setminus \{i\}}
                   \end{cases}
    .
  \end{align*}

  For each $i\in \underline{n}$ and $\star\in\pm^n$ let $\star^i\in \pm^n$ be such that $\star^i_j=\star_j$ if $i\neq j$ and $\star^i_i=-\star_i$. If $S\subset \underline{n}$ and $i\in S$ then
  \begin{align*}
    d^-_i\left(\spec[\star,S,\emptyset]{n}\right)
     & =\left(\push[-\star_{\underline{n}\setminus S}]{\underline{n}\setminus S}\pr_{\underline{n}\setminus S}
    \underset{\substack{\star'\in \pm^{\underline{n}\setminus S} \\\star'\neq \star_{\underline{n}\setminus S}}}{\bigcirc}\pro{S}\right)
    \circ_{\star_{\underline{n}\setminus S}}
    \left(\push[\star_{S\setminus\{i\}}]{S\setminus\{i\}}\pr_{S\setminus\{i\}}
    \underset{\star''\in \pm^{S\setminus\{i\}}}{\bigcirc}\pro{\{i\}}\right) \\
     & =d^-_i\left(\spec[\star^i,S,\emptyset]{n}\right); \\
    d^+_i\left(\spec[\star,S,\emptyset]{n}\right)
     & =\left(\push[-\star_{\underline{n}\setminus S}]{\underline{n}\setminus S}\pr_{\underline{n}\setminus S}
    \underset{\star'\neq \star_{\underline{n}\setminus S}}{\bigcirc}\pro{S}\right)
    \circ_{\star_{\underline{n}\setminus S}}
    \left(\pro{\{i\}}
    \begin{cases}\circ_{-\star_i}\pro{S\setminus\{i\}} \\
      \circ_{\star_i}\push[\star_{S\setminus\{i\}}]{S\setminus\{i\}}\pr_{S\setminus\{i\}}
    \end{cases}\right) \\
     & =d^-_i\left(\spec[\star,S\setminus\{i\},\emptyset]{n}\right).
  \end{align*}

  If $S,T\subset \underline{n}$ are disjoint, $T\neq \emptyset$ and $i\in S$ then
  \begin{align*}
    d^-_i\left(\spec[\star,S,T]{n}\right)
     & =\left(\pro{T}
    \underset{\substack{\star'\in\pm^T \\\star'\neq \star_T,-\star_T}}{\bigcirc}\pro{\underline{n}\setminus T}
    \right)\begin{cases}
             \circ_{-\star_T}
             \push[-\star_{\underline{n}\setminus (S\sqcup T)}]{\underline{n}\setminus (S\sqcup T)}\pr_{\underline{n}\setminus (S\sqcup T)}
             \underset{\star''\in\pm^{\underline{n}\setminus(S\sqcup T)}}{\bigcirc}\pro{S} \\
             \circ_{\star_T}\left(
             \push[\star_{\underline{n}\setminus(\{i\}\cup T)}]{\underline{n}\setminus(\{i\}\cup T)}\Phi_{S,T}
             \underset{\star'''\in\pm^{\underline{n}\setminus(\{i\}\cup T)}}{\bigcirc}
             \pro{\{i\}}\right)
           \end{cases}
    \\
     & =d^-_i\left(\spec[\star^i,S,T]{n}\right);   \allowdisplaybreaks \\
    d^+_i\left(\spec[\star,S,T]{n}\right)
     & =\left(\pro{T}
    \underset{\substack{\star'\in \pm^T \\\star'\neq \star_T,-\star_T}}{\bigcirc}\pro{\underline{n}\setminus T}
    \right)\begin{cases}
             \circ_{-\star_T}
             \left(\push[-\star_{\underline{n}\setminus (S\sqcup T)}]{\underline{n}\setminus (S\sqcup T)}\pr_{\underline{n}\setminus (S\sqcup T)}
             \underset{\star''\in\pm^{\underline{n}\setminus(S\sqcup T)}}{\bigcirc}\pro{S}
             \right) \\
             \circ_{\star_T}
             \left(
             \push[\star_{\underline{n}\setminus(\{i\}\cup T)}]{\underline{n}\setminus(\{i\}\cup T)}\Phi_{S\setminus\{i\},T}
             \underset{\star'''\in\pm^{\underline{n}\setminus (\{i\}\cup T)}}{\bigcirc}
             \pro{\{i\}}
             \right)
           \end{cases}
    \\
     & =d^-_i\left(\spec[\star,S\setminus\{i\},T]{n}\right).
  \end{align*}

  If $S,T\subset \underline{n}$ are disjoint and $i\in T$ then
  \begin{align*}
    d^-_i\left(\spec[\star,S,T]{n}\right)
     & =\left(\pro{T}
    \underset{\substack{\star'\in \pm^T \\\star'\neq \star_T,-\star_T}}{\bigcirc}\pro{\underline{n}\setminus T}
    \right)\begin{cases}
             \circ_{-\star_T}
             \left(
             \push[-\star_{\underline{n}\setminus(S\sqcup T)}]{\underline{n}\setminus(S\sqcup T)}\pr_{\underline{n}\setminus(S\sqcup T)}
             \underset{\star''\in \pm^{\underline{n}\setminus(S\sqcup T)}}{\bigcirc}\pro{S}
             \right) \\
             \circ_{\star_T}
             \push[\star_{\underline{n}\setminus T}]{\underline{n}\setminus T}\Phi_{S,T\setminus\{i\}}
           \end{cases}
    \\
     & =d^+_i\left(\spec[\star,S,T\setminus\{i\}]{n}\right); \\
    d^+_i\left(\spec[\star,S,T]{n}\right)
     & =\left(\pro{T}
    \underset{\substack{\star'\in \pm^T \\\star'\neq \star_T,-\star_T}}{\bigcirc}\pro{\underline{n}\setminus T}
    \right)
    \begin{cases}
      \circ_{-\star_T}
      \left(
      \push[-\star_{\underline{n}\setminus(S\sqcup T)}]{\underline{n}\setminus(S\sqcup T)}\pr_{\underline{n}\setminus(S\sqcup T)}
      \underset{\star''\in\pm^{\underline{n}\setminus(S\sqcup T)}}{\bigcirc}\pro{S}
      \right) \\
      \circ_{\star_T}
      \left(
      \push[\star_{\underline{n}\setminus (S\sqcup T)}]{\underline{n}\setminus (S\sqcup T)}\pr_{\underline{n}\setminus(S\sqcup T)}
      \underset{\star'''\in \pm^{\underline{n}\setminus(S\sqcup T)}}{\bigcirc}\pro{S}
      \right)
    \end{cases} \\
     & =d^+_i\left( \spec[-\star,S,T]{n} \cdot(\colc_1\colc_2)\right).
  \end{align*}

  We may thus define, by abuse of notation,
  \begin{gather*}
    \spec[+]{n}:[-1,1]^n\rightarrow \SCv(2,\pm^n)\\
    \spec[+]{n}(t)=\spec[\sgn(t),\{i\in\underline{n}\mid \lvert t_i\rvert\leq 1/3\},\{i\in\underline{n}\mid \lvert t_i\rvert\geq 2/3\}]{n}\left(\begin{cases}
        \sgn(t_i)6t_i-1, & \lvert t_i\rvert\leq 1/3         \\
        \sgn(t_i)6t_i-3, & 1/3\leq \lvert t_i\rvert\leq 2/3 \\
        \sgn(t_i)6t_i-5, & 2/3\leq \lvert t_i\rvert
      \end{cases}\right)_{i\in\underline{n}}.
  \end{gather*}
  The equations we verified above guarantee that the points where this formula is ambiguous, ie points with some coordinate an integer multiple of $\frac{1}{3}$, are uniquely determined when we extend by continuity on the unambiguous points. For all $t\in \partial [-1,1]^n$ we have $\spec[+]{n}(t)=(\spec[+]{n}\cdot(\colc_1\colc_2))(-t)$. Thus we have that $\spec{n}$ is closed.

  Let $\gam[+]{n}\in C_{n+1}(\SCv_{n+1}(2,\pm^n))$ be defined as
  \begin{align*}
    \gam[+]{n}(t)_{\colc_i} & \coloneqq\min(1,1-t_{n+1})\spec[+]{n}(t_{\underline{n}})_{\colc_i}+\max(0,t_{n+1})\act{n}(\loo[+]{n}(t_{\underline{n}}),\pro{n}),
    \\
    \gam[+]{n}(t)_{\star}   & \coloneqq\max(-t_{n+1},0)\spec[+]{n}(t_{\underline{n}})_{\star}+\left[0,\frac{t_{n+1}+1}{4}\right]\times\prod_{i=1}^n\left[\frac{\star_i-1}{2},\frac{\star_i+1}{2}\right].
  \end{align*}
  The following conditions guarantee that for all $t\in [-1,1]^{n+1}$ the cubes in the configurations $\gam[+]{n}(t)$ have pairwise disjoint interiors, and so are indeed elements of $\SCv(2,\pm^n)$:
  \begin{enumerate}
    \item In the first half of the homotopy the closed cubes stay constant and all open cubes are deformed in the $0$th coordinate so that halfway through the homotopy they are separated from the closed cubes by the hyperplane $\{\frac{1}{4}\}\times [-1,1]^n$;

    \item If $t\in\left[-\frac{1}{2},\frac{1}{2}\right]^n$ then $\spec[+]{n}(t)_{\colc_1}$ and $\spec[+]{n}(t)_{\colc_2}$ are separated by $\left\{\frac{1}{2}\right\}\times[-1,1]^n$;

    \item If $t\in\left[-\frac{1}{2},\frac{1}{2}\right]^n$ then $\act{n}(\loo[+]{n}(t),\pro{n})_{\colc_1}$ and $\act{n}(\loo[+]{n}(t),\pro{n})_{\colc_2}$ are separated by $\left\{\frac{3}{4}\right\}\times[-1,1]^n$;

    \item If $i\in\underline{n}$ and $\lvert t_i\rvert\geq \frac{1}{2}$ then $\spec[+]{n}(t)_{\colc_1}$ and $\act{n}(\loo[+]{n}(t),\pro{n})_{\colc_1}$ are in the same side of the hyperplane $[0,1]\times[-1,1]^{i-1}\times\{0\}\times [-1,1]^{n-i}$, and $\spec[+]{n}(t)_{\colc_2}$ and $\act{n}(\loo[+]{n}(t),\pro{n})_{\colc_2}$ are in the other.
  \end{enumerate}
  Thus we have that
  $$
    d^-_{n+1}\gam[+]{n}=\spec[+]{n}
    ,\qquad
    d^+_{n+1}\gam[+]{n}=\act{n}(\loo[+]{n},\pro{n}).
  $$
  By construction for all $t\in\partial[-1,1]^n\times[-1,1]$ we have $\gam[+]{n}(t)=\gam[+]{n}(-t_{\underline{n}},t_{m+1})\cdot (\colc_1\colc_2)$.
  It follows that for
  \begin{equation}\label{eq gamn}
    \gam{n} \coloneqq \gam[+]{n}-(-1)^n\gam[+]{n}\cdot (\colc_1\colc_2),
  \end{equation}we have
  \begin{equation*}
    d\gam{n}
    =\sum_{i=1}^{n+1}(-1)^{i-1}(d^+_i\gam{n}-d^-_i\gam{n})
    =(-1)^n(\act{n}(\loo{n},\pro{n})-\spec{n})
    \qedhere
  \end{equation*}
\end{proof}

\begin{lemma}\label{Lemma NonForm Dim n}
  If we are given cubical chains:
  \begin{align*}
    \pro*{S}
     & \in C_0(\SCv_{n+1}(0, {\pm^S})),
     & \act*{n}
     & \in C_0(\SCv_{n+1}(1,1)),
     & \loo*{n}
     & \in C_n(\mathcal C_{n+1}(2)), \\
    \pro-{S}
     & \in C_1(\SCv_{n+1}(0, {\pm^S})),
     & \act-{n}
     & \in C_1(\SCv_{n+1}(1,1)),
     & \loo-{n}
     & \in C_{n+1}(\mathcal C_{n+1}(2))
  \end{align*}
  for $S\subset \underline{n}$, such that:
  \begin{align*}
    \pro*{S}
     & = \pro{S}+ d  \pro-{S},
     & \act*{n}
     & = \act{n}+ d  \act-{S},
     & \loo*{n}
     & = \loo{n}+ d \loo-{n},
  \end{align*}
  then there are chains
  \begin{align*}
    \push*[\star]{S}
     & \in C_{\lvert S\rvert}(\SCv_{n+1}(1,\pm^S)),
     & \push-[\star]{S}
     & \in C_{\lvert S\rvert+1}(\SCv_{n+1}(1,\pm^S)),
     & \gam*{n}
     & \in C_{n+1}(\SCv_{n+1}(2,\pm^n))
  \end{align*}
  for each $S\subset \underline{n}$ and $\star\in \pm^S$ such that $\push*[ ]{\emptyset}=\act*{n}$, $\push-[ ]{\emptyset}=\act-{n}$ and, for $\spec*[]{n} \in C_n(\SCv_{n+1}(2,\pm^n))$ defined as in definition \ref{def spec n}, we have
  \begin{align}
    \label{boundary tilde push}
    d\push*[\star]{S}
    = & \sum_{i\in S}
    (-1)^{i}\left(
    \push*[\star_{S\setminus \{i\}}]{S\setminus \{i\}}\pr_{S\setminus\{i\}}\underset{\star'\in \pm^{S\setminus \{i\}}}{\bigcirc}\pro*{\{i\}}
    -\pro*{\{i\}}\begin{cases}
                     \circ_{-\star_i}\pro*{S\setminus\{i\}} \\
                     \circ_{\star_i}\push*[\star_{S\setminus\{i\}}]{S\setminus\{i\}}\pr_{S\setminus\{i\}}
                   \end{cases}
    \right.\nonumber \\
      &
    +\sum_{j\in S\setminus\{i\}}(-1)^j
    \Delta\left(\left(\push-[\star_{S\setminus \{i,j\}}]{S\setminus \{i,j\}}\pr_{S\setminus\{i,j\}}
      \underset{\star'\in \pm^{S\setminus \{i,j\}}}{\bigcirc}\pro-{\{j\}}
      \right)
      \underset{\star''\in \pm^{S\setminus \{i\}}}{\bigcirc}\pro-{\{i\}}
    \right.\nonumber \\
      &
      -\left(\pro-{\{j\}}\begin{cases}
                           \circ_{-\star_j}\pro-{S\setminus\{i,j\}} \\
                           \circ_{\star_j}\push-[\star_{S\setminus \{i,j\}}]{S\setminus \{i,j\}}\pr_{S\setminus\{i,j\}}
                         \end{cases}
      \right)
      \underset{\star'\in \pm^{S\setminus \{i\}}}{\bigcirc}\pro-{\{i\}}
    \nonumber \\
      & -\pro-{\{i\}}\begin{cases}
                         \circ_{-\star_i}\pro-{S\setminus\{i\}} \\
                         \circ_{\star_i}\left(\push-[\star_{S\setminus \{i,j\}}]{S\setminus \{i,j\}}\pr_{S\setminus\{i,j\}}\underset{\star'\in \pm^{S\setminus \{i,j\}}}{\bigcirc}\pro-{\{j\}}\right)
                       \end{cases}
    \nonumber \\
      & \left.\left.+
      \pro-{\{i\}}\begin{cases}
                    \circ_{-\star_i}\pro-{S\setminus\{i\}} \\
                    \circ_{\star_i}
                    \left(\pro-{\{j\}}\begin{cases}
                            \circ_{-\star_j}\pro-{S\setminus\{i,j\}} \\
                            \circ_{\star_j}
                            \push-[\star_{S\setminus\{i,j\}}]{S\setminus\{i,j\}}\pr_{S\setminus\{i,j\}}
                          \end{cases}
                    \right)
                  \end{cases}\right)
    \right) \\
    \label{eq del gamma}
    d \gam*{n}
    = & \ \spec*[]{n}-\act*{n}(\loo*{n},\pro*{n}).
  \end{align}
\end{lemma}

\begin{proof}
  We will define recursively (by induction on the cardinality of $S$) chains
  \begin{align*}
    \push*[\star]{S}
     & \in C_{\lvert S\rvert}(\SCv_{n+1}(1,\pm^S)),
     & \push-[\star]{S}
     & \in C_{\lvert S\rvert+1}(\SCv_{n+1}(1,\pm^S)).
  \end{align*}
  For $S=\emptyset$ we simply set $\push*[]{\emptyset}\coloneqq \act*{n}$ and $\push-[]{\emptyset}\coloneqq \act-{n}$.
  If $S\subset \underline n$ is non-empty and $\star\in\pm^S$ then we set
  \begin{align*}
    \push*[\star]{S}
     & \coloneqq \push[\star]{S}+\sum_{i\in S}(-1)^{i}\Delta\left(
    \push-[\star_{S\setminus \{i\}}]{S\setminus \{i\}}\pr_{S\setminus \{i\}}\underset{\star'\in \pm^{S\setminus \{i\}}}{\bigcirc}\pro-{\{i\}}
    -\pro-{\{i\}}\begin{cases}
                   \circ_{-\star_i}\pro-{S\setminus\{i\}} \\
                   \circ_{\star_i}\push-[\star_{S\setminus\{i\}}]{S\setminus\{i\}}\pr_{S\setminus \{i\}}
                 \end{cases}\!\!\!\!\!\!
    \right) \\
    \push-[\star]{S}
     & \coloneqq
    \push[\star]{S}\pr_S+\sum_{i\in S}(-1)^{i}\Delta\left(
    \push-[\star_{S\setminus \{i\}}]{S\setminus \{i\}}\pr_{S\setminus \{i\}}\underset{\star'\in \pm^{S\setminus \{i\}}}{\bigcirc}\pro-{\{i\}}
    -\pro-{\{i\}}
    \begin{cases}
      \circ_{-\star_i}\pro-{S\setminus\{i\}}
      \\
      \circ_{\star_i}\push-[\star_{S\setminus\{i\}}]{S\setminus\{i\}}\pr_{S\setminus \{i\}}
    \end{cases}\!\!\!\!\!\!
    \right)\Psi_{\lvert S\rvert}
  \end{align*}
  such that equation \ref{boundary tilde push} holds. Setting
  \begin{align*}
    \spec-[\star,S,\emptyset]{n}
              & \coloneqq
    \Biggl(
    \push-[-\star_{\underline{n}\setminus S}]{\underline{n}\setminus S}\pr_{\underline{n}\setminus S}
    \underset{\substack{\star'\in \pm^{\underline{n}\setminus S}
    \\
        \star'\neq \star_{\underline{n}\setminus S}}}{\bigcirc}\pro-{S}
    \Biggr)
    \circ_{\star_{\underline{n}\setminus S}}\push-[\star_{S}]{S}\pr_S
    ;
    \\
    \spec-[\star,S,T]{n}
              & \underset{\mathclap{T \neq \emptyset}}{\coloneqq}
    \Bigl(
    \pro-{T}\underset{\substack{\star'\in\pm^T \\ \star'\neq \star_T,-\star_T}}{\bigcirc} \pro-{\underline{n}\setminus T}
    \Bigr)\begin{cases}
            \circ_{-\star_T}
            \push-[-\star_{\underline{n}\setminus (S\sqcup T)}]{\underline{n} \setminus (S\sqcup T)}\pr_{\underline{n} \setminus (S\sqcup T)}
            \underset{\star''\in\pm^{\underline{n}\setminus (S\sqcup T)}}{\bigcirc}\pro-{S} \\
            \circ_{\star_T}
            \push-[\star_{\underline{n}\setminus T}]{\underline{n}\setminus T}\Phi_{S,T}
          \end{cases}
    ; \\
    \spec-{n} & \coloneqq \left(\sum_{\star\in \pm^n}\sgn\star
    \sum_{\substack{S,T\subset \underline{n} \\ S\cap T=\emptyset}}
    \spec-[\star,S,T]{n}\right)\cdot(1-(-1)^n(\colc_1\colc_2)); \\
    \gam*{n}  & \coloneqq \gam{n}+\Delta(\spec-{n}) -\Delta(\act-{n}(\loo-{n},\pro-{n}))
  \end{align*}
  the equation \eqref{eq del gamma} holds.
\end{proof}

From the previous lemmas and the same argument as in the proof of Theorem~\ref{NonFormality n=1} the general non-formality result follows.
Note that to mimic the proof of Theorem~\ref{NonFormality n=1}, we need to make sense of elements of the form $x \Phi_{S, T}$ in the intermediate operad $\opQ$.
Thanks to Corollary~\ref{cor eqv cub operads}, we may assume that $\opQ$ is given by normalized chains of a relative operad in cubical $\omega$-groupoids.

\begin{theorem}
  Voronov's Swiss-Cheese operad $\SCv_{n+1}$ is not formal over any field of characteristic different from 2 for any $n\geq 1$.
\end{theorem}

\begin{proof}
  Assume by contradiction there there is a relative dg operad $\mathcal{Q}$ and a span of quasi-isomorphisms
  $$\begin{tikzcd}
      C_\bullet(\SCv_{n+1}) & \mathcal{Q} \arrow[l, "\sim", "\phi" swap] \arrow[r, "\sim" swap, "\psi"] & H_\bullet(\SCv_{n+1})
    \end{tikzcd}
    .$$
  Since $\phi$ is a quasi-isomorphism there are cycles
  $$
    \pro{Q}{}^{,S} \in \mathcal{Q}(0, {\pm}^S)_0,
    \qquad \act{Q} \in \mathcal{Q}(1,1)_0,
    \qquad \loo{Q} \in \mathcal{Q}(2)_n
  $$
  for $S\subset \underline n$, and
  $$
    \pro-{S} \in C_1(\SCv_2(0, {\pm^S})),
    \qquad \act-{n} \in C_1(\SCv_2(1,1)),
    \qquad \loo-{n} \in C_{n+1}(\SCv_{n+1}(2))
  $$
  such that
  $$
    \phi \pro{Q}{}^{,S} = \pro{S}+ d \pro-{S}
    ,\qquad \phi \act{Q} = \act{n}+ d \act-{n}
    ,\qquad \phi \loo{Q} = \loo{n}+ d \loo-{n}.
  $$
  Let $\push*[\star]{S}$, $\push-[\star]{S}$ and $\gam*{n}$ be the chains in the conclusion of lemma \ref{Lemma NonForm Dim n}.

  Using Lemma \ref{Lemma NonForm Dim n} and the fact that $\phi$ is a quasi-isomorphism we can show there are chains $\push[\star]{Q}{}^{,S}\in\opQ(1,\pm^S)_{\lvert S\rvert}$ such that
  $$
    d\phi(\push[\star]{Q}{}^{,S})
    =d\push*[\star]{S}.
  $$

  Define $\spec{Q}\in \mathcal{Q}(2,\pm^n)_{n}$ as in definition \ref{def spec n}, which we can due to Corollary \ref{cor eqv cub operads}.  Due to equation \ref{boundary tilde push} we have, by an analogous argument as the one for Lemma \ref{Lemma Spec is closed an homologous to loop}, that
  $$
    d\spec{Q}=0.
  $$

  The group $H_{n}(\SCv_{n+1}(2,\pm^n))$ is generated by $[\act{n}(\loo{n},\pro{n})] = [\phi\act{Q}(\phi \loo{Q},\phi\pro{Q})]$, which implies that $H_n(\mathcal{Q}(2,\pm^n))$ is generated by $[\act{Q}( \loo{Q},\pro{Q})]$ and that there are
  $$
    \gam{Q}\in\mathcal{Q}(2,\pm^n)_{n+1},
    \qquad \lambda\in \K
  $$
  such that
  \begin{equation}\label{Eq NF dim n}
    \spec{Q}=\lambda \act{Q}( \loo{Q},\pro{Q})+d\gam{Q}.
  \end{equation}

  By an argument analogous to the one in the proof of theorem \ref{NonFormality n=1} we have
  \begin{align*}
    \spec*[]{n}
     & =\phi\act{Q}(\phi \loo{Q},\phi\pro{Q})+d\gam*{n},
  \end{align*}
  and since $[\phi\act{Q}(\phi \loo{Q},\phi\pro{Q})]$ generates $H_n(\SCv_{n+1}(2,\pm^n))$ we conclude that $\lambda=1$.

  As in the proof of theorem \ref{NonFormality n=1}, for degree reasons and the fact that $\psi$ is a quasi-isomorphism we can derive $\lambda=0$, a contradiction.
\end{proof}

\addsec{Table of notation}

\renewcommand{\arraystretch}{1.25}

\begin{center}
  \begin{tabular}{lllll}
    \toprule
    Notation
     & Element of
     & Definitions
    \\
     &
     & $n=1$
     & $n=2$
     & all $n$
    \\ \midrule
    $\pro[]{n}$
     & $\SCv_{n+1}(0, \pm^n)$
     & \ref{def pro1}
     & \ref{def pro2}
     & \ref{def proN}
    \\
    $\inc[S]{n}$
     & $\hom(\SCv_{|S|}, \SCv_{n+1})$
     & n/a
     & \ref{def inc2}
     & \ref{def inc}
    \\
    $\act{n}$
     & $\SCv_{n+1}(1, 1)$
     & \ref{def act1}
     & \ref{def act2}
     & \ref{def actn}
    \\
    $\push[\star]{n}$
     & $C_{n}(\SCv_{n+1}(1, \pm^n))$
     & \ref{def push1}
     & \ref{def push2}
     & \ref{def pushn}
    \\
    $\loo{n}$
     & $C_n(\C_{n+1}(2))$
     & \ref{def loo1}
     & n/a
     & \eqref{eq loon}
    \\
    $\spec[\star]{n}$
     & $C_n(\SCv_{n+1}(2, \pm^n))$
     & \ref{def spec1}
     & n/a
     & \eqref{eq specn}
    \\
    $\gam{n}$
     & $C_{n+1}(\SCv_{n+1}(2, \pm^n))$
     & \ref{lem gam1}
     & n/a
     & \eqref{eq gamn}
    \\ \bottomrule
  \end{tabular}
\end{center}

\printbibliography

\end{document}